\documentclass[a4paper]{article}

\usepackage[all]{xy}\usepackage[latin1]{inputenc}        %accents
\usepackage[dvips]{graphics,graphicx}
\usepackage{amsfonts,amssymb,amsmath,color,mathrsfs, amstext,bm}
\usepackage{amsbsy, amsopn, amscd, amsxtra, amsthm,authblk, enumerate}
\usepackage{upref}
\usepackage[colorlinks,
            linkcolor=red,
            anchorcolor=red,
            citecolor=red
            ]{hyperref}

\usepackage{geometry}
\usepackage[displaymath]{lineno}
\usepackage{float}

\usepackage{yhmath}

\usepackage[normalem]{ulem}

\usepackage{algorithm}
\usepackage{algpseudocode}

\usepackage{algorithm}
\usepackage{algpseudocode}
\usepackage{caption}
\usepackage{xcolor}

\numberwithin{equation}{section}

\newcommand\keywords[1]{\textbf{Keywords:} #1}

\newtheorem*{theorem*}{Theorem}

\newtheorem{theorem}{Theorem}[section]
\newtheorem{lemma}{Lemma}[section]
\newtheorem{assumption}{Assumption}[section]

\newtheorem{remark}{Remark}[section]

\begin{document}

\title{Convergence of Random Batch Method with replacement for interacting particle systems}
%\title{Kinetic Monte Carlo meets interacting particle systems: Random Batch Method algorithms and convergence}
\author[a]{Zhenhao Cai\thanks{E-mail:caizhenhao@pku.edu.cn}}
\author[b]{Jian-Guo Liu\thanks{E-mail:jliu@math.duke.edu}}
\author[c]{Yuliang Wang\thanks{Email:YuliangWang$\_$math@sjtu.edu.cn}}
\affil[a]{School of Mathematical Sciences, Peking University, Beijing, 100871, P.R.China.}
\affil[b]{Department of Mathematics, Department of Physics, Duke University, Durham, NC 27708, USA.}
\affil[c]{School of Mathematical Sciences, Institute of Natural Sciences, Shanghai Jiao Tong University, Shanghai, 200240, P.R.China.}

\date{}
\maketitle

%    Abstract is required.
\begin{abstract}
The Random Batch Method (RBM) proposed in [Jin et al. J Comput Phys, 2020] is an efficient algorithm for simulating interacting particle systems (IPS). In this paper, we investigate the Random Batch Method with replacement (RBM-r), which is the same as the kinetic Monte Carlo (KMC) method for the pairwise interacting particle system of size $N$. In the RBM-r algorithm, one randomly picks a small batch of size $p \ll N$, and only the particles in the picked batch interact among each other within the batch for a short time, where the weak interaction (of strength $\frac{1}{N-1}$) in the original system is replaced by a strong interaction (of strength $\frac{1}{p-1}$). Then one repeats this pick-interact process.
This KMC algorithm dramatically reduces the computational cost from $O(N^2)$ to $O(pN)$ per time step, and provides an unbiased approximation of the original force/velocity field of the interacting particle system.
We give a rigorous proof of this approximation with an explicit convergence rate. In detail, we show that the Wasserstein-2 distance between first marginal distributions of IPS and RBM-r has an $O(\kappa^{1/4})$ upper bound, where $\kappa$ is the time step for choosing the random batch and the bound is independent of $N$. An improved $O(\kappa^{1/2})$ rate is also obtained when there is no diffusion in the system. Notably, the techniques in our analysis can potentially be applied to study KMC for other systems, including the stochastic Ising spin system.

\end{abstract}


\maketitle

\section{Introduction}\label{sec:intro}

Many systems in natural and social sciences including molecules in fluids, plasma, swarming, chemotaxis, flocking, synchronization, consensus, random vortex model, can be effectively modeled by $N$ indistinguishable individuals called particles. The particles interact with each other according to some interaction force. For the pairwise interacting case, direct evaluation of the interaction requires complexity of $O(N^2)$ per time step. The Random Batch Method (RBM) proposed in \cite{jin2020random}, as a plausible approximation of the interacting particle system, has a much cheaper computational cost $O(pN)$, where $p \ll N$ represents the batch size, as explained below. Meanwhile, the proposed RBM methods can be proved to converge to the original interacting particle system (IPS), in the sense of distributions (see Theorem 3.1 in \cite{jin2020random} and Theorem \ref{eq:thm:mainthm} in this paper). The authors of \cite{jin2020random} proposed two versions of RBM: RBM without replacement (RBM-1) and RBM with replacement (RBM-r). In RBM-1, the index set $\{1, \dots, N\}$ is randomly divided into $\tfrac{N}{p}$ batches, each with size $p$. The particles then pairwisely interact with each other only within their own batches for a short time and then one randomly reshuffles and continues. In RBM-r, one randomly picks a batch of size $p$, and only the particles in the picked batch interact among each other within the batch for a short time, with the interaction strength increased from $\frac{1}{N-1}$ to $\frac{1}{p-1}$. Then one repeats this pick-interact process. This pick-interact algorithm (or RBM-r) is exactly the kinetic Monte Carlo (KMC) method \cite{bortz1975new, voter2007introduction} for the first-order pairwise interacting particle system. The main purpose of this paper is to study the convergence of RBM-r via a sequence of probabilistic tools. The method we developed here shall also be able to study other KMC methods in statistical physics \cite{bortz1975new, binder1973scaling, voter2007introduction}.

% The two methods differ in the way of choosing the random batch: letting $N$ be the particle number and $p$ be the batch size, RBM-1 randomly derive $\{1,\dots,N\}$ into $N/p$ batches each with size $p$.
% Particles interacting each other only within their own  batches for a short time and then randomly reshaffing. While RBM-r picks a batch of a size $p$ randomly from $\{1,\dots,N\}$ with replacement each time. Particles interact with each other only within the picked batch for a short time and then start over again. This is exactly a kinetic Monte Carlo (KMC) method applied to the interacting particle system (?). While the convergence of RBM-1 is well-studied, the convergence of the other version, RBM-r, is still open. The main purpose of this paper is to study the convergence of RBM-r via a sequence of probabilistic tools. The method we developed here shall also be used to study the converegnce of other KMC methods in statistical physics \cite{bortz1975new, binder1973scaling, voter2007introduction}. 

We begin with a brief description of the interacting particle systems and the main idea of the Random Batch Method. Consider the following first-order pairwise interacting particle system in $\mathbb{R}^d$, described by a stochastic differential equation (SDE) system in It\^o's sense:
\begin{equation}\label{eq:IPSintro0}
dX^i = -\nabla V(X^i) dt + \frac{1}{N-1}\sum_{j\neq i}K\left(X^j - X^i \right) dt + \sigma dB^i,\quad X^i|_{t=0} = X_0^i, \quad 1\leq i \leq N,
\end{equation}
where $X^i \in \mathbb{R}^d$ ($1\leq i \leq N$), $V: \mathbb{R}^d \rightarrow \mathbb{R}$ is the external potential, $K: \mathbb{R}^d \rightarrow \mathbb{R}^d$ is the interaction kernel, $B^i$ ($1 \leq i \leq N$) are independent standard Brownian motions in $\mathbb{R}^d$, and $\sigma \geq 0$ is the constant volatility. Note that the $\frac{1}{N-1}$ factor for each pairwise interaction force in $\eqref{eq:IPSintro0}$ is the kinetic scale. This force $\frac{1}{N-1}K(X^j - X^i)$ is a weak interaction since usually $N$ is a large number. As we would see later, in both versions of the Random Batch Method, one replaces the weak interaction force $\frac{1}{N-1}K(X^j - X^i)$ between all pairs $(X^j,X^i)$ with some strong interaction force $\frac{1}{p-1}K(X^j - X^i)$ within pairs selected in a batch with size $p$ (usually $p$ is a small number). In other words, the main modification in the RBM is replacing  $\frac{1}{N-1}\sum_{j\neq i}K\left(X^j - X^i \right)$ (the total force acting on $X^i$) with $\frac{1}{p-1}\sum_{j\in \mathcal{C},j\neq i}K\left(X^j - X^i \right)$, where $\mathcal{C} \subset \{1,\dots,N\}$ is a random batch of size $p$. 
This replacement in the algorithm dramatically reduces the computational cost from $O(N^2)$ to $O(pN)$ per time step, and provides an unbiased approximation for the force / velocity field of the original interacting particle system.
% Each total force $\frac{1}{N-1}\sum_{j\neq i}K\left(X^j - X^i \right)$ and $\frac{1}{p-1}\sum_{j\in \mathcal{C},j\neq i}K\left(X^j - X^i \right)$ has effect of the same scale to the system. 
i.e. for any fixed deterministic sequence $(x^i)_{i=1}^N$, 
\begin{equation}\label{eq:unbiased}
\mathbb{E}\big[\frac{1}{p-1}\sum_{j\in \mathcal{C},j\neq i}K\left(x^j - x^i \right) \mid i \in \mathcal{C}\big] = \frac{1}{N-1}\sum_{j\neq i}K\left(x^j -x^i \right),
\end{equation}
where the randomness comes from the random batch $\mathcal{C}$. This property \eqref{eq:unbiased} is crucial in our analysis, in particular, in proving Lemma \ref{eq:deltakai} below. This consistency property \eqref{eq:unbiased} indicates that the modified dynamics in both RBM-1 and RBM-r should be close to the original interacting particle system.

As a further remark, the applicability of RBM is not limited to the system \eqref{eq:IPSintro0}. For instance, the particle interaction here in \eqref{eq:IPSintro0} is of the binary form, which is actually more important and applicable compared with other forms like Kac interaction in spin glasses \cite{frohlich1987some, franz2004finite}. Typical examples of such binary interactions include Coulomb's interaction, and influence of opinions in the stochastic opinion dynamics \cite{motsch2014heterophilious}. However, as discussed in \cite{jin2020random}, RBM can also be applied to non-binary interacting particle systems. Furthermore, in \eqref{eq:IPSintro0}, the particles we consider are indistinguishable, but RBM can actually be applied to systems with unequal weights, charges or volatility \cite{jin2021convergence}. Moreover, we have assumed the constant volatility $\sigma$ for simplicity to illustrate our methodology and analysis. Actually, RBM can also be applied to systems with degenerate or multiplicative noise and has similar analysis. It can also be extended to jumping processes like the stochastic Ising model \cite{bortz1975new}. 
% As a further remark, the proposed Random Batch Method only works for weakly interacting systems, since here we only consider the system with interaction of the form $\frac{1}{N-1}\sum_{j\neq i}K\left(X^j - X^i \right)$. This is indeed a weak interacting force since $N$ is usually a very large number. As we would see in the formulation of the Radom Batch Method below, during each small time duration, the interacting force within the corresponding random group of particles is a strong interaction, of the form $\frac{1}{p-1}\sum_{j\neq i, j \in C}K\left(X^j - X^i \right)$. Here $C$  is the random batch and the small number $p = |C|$ is the batch size to be introduced below. Such formulation of strong interaction in RBM can converge to the original weakly interacting particle system, due to the certain time average in time, and thus the convergence is like the convergence in the law of large number (in time) (see \cite{jin2020random, jin2022mean} for more details).

Next, we describe the motivations and formulations of RBM-1 and RBM-r. Consider the direct simulation of the interacting particle system \eqref{eq:IPSintro0}. Take time step $\kappa>0$. Suppose we compute up to time $T$ with a total of $n_T := T/\kappa$ time intervals, then clearly the total complexity is $O(n_T N^2)$, which is very expensive for large $N$. Inspired by the ``mini-batch" idea in the celebrated Stochastic Gradient Descent (SGD) algorithm in machine learning \cite{robbins1951stochastic}, the authors of \cite{jin2020random} proposed the Random Batch Method (RBM) to reduce computational cost. As we will see below, the computational cost of RBM is only $O(n_T pN)$, where the batch size $p$ is usually a small number and often equals 2. This $O(N)$ complexity can be compared with the well-known fast multipole method (FMM) \cite{greengard1987fast}, which has the computational cost of the same order and is more accurate than RBM. However, RBM is simpler to implement and is valid for systems with more general potentials that do not have the decay property, unlike the Coulomb's potential required by FMM. To be more specific, the key features of FMM are more complicated, mainly including the hierarchical subdivision of space, the far-field expansion of the kernel in which the influence of source and evaluation points separates, and the conversion of far field expansions into local expansions. As we would see in the RBM algorithms described below, the implementation of RBM is much simpler and does not have much restrictions on the interaction kernel and the computational domain.
The idea of RBM is to use random batch (of size $p$) on the summation of the interaction force in \eqref{eq:IPSintro0}, as discussed near \eqref{eq:unbiased}. 
% Namely, at each time, one only picks $p$ particles at random and considers the interacting system among these $p$ particles for a small period of time $\kappa$. At first glance, after only a few such operations, the resulted dynamics would differ a lot from the original dynamics \eqref{eq:IPSintro0}. However, once we run such operations for infinitely many times, the resulted dynamics would be close to the original one due to the law of large number (in time). In other words, in some fixed time interval, if we randomly pick a batch at each $t_k := k\kappa$ ($k = 0,1,2,\dots$) for some $\kappa > 0$, the resulted dynamics would converge to \eqref{eq:IPSintro0} as $\kappa \rightarrow 0$.
After picking time step $\kappa > 0$, one natural question is how we in practice choose the random batch $\mathcal{C}_k$ at each time grid $t_k := k\kappa$ ($k = 0,1,2,\dots$). The authors of \cite{jin2020random} proposed two approaches: the Random Batch Method without replacement (RBM-1) and the Random Batch Method with replacement (RBM-r). As introduced at the beginning, the idea of RBM-1 is that for $k=0,1,2,\dots$, at each time $t_k$, we randomly divide $\{1,\dots,N\}$ into $\tfrac{N}{p}$ small batches, each with size $p$ ($p \ll N$, often $p=2$; for simplicity, we assume throughout the paper that $p$ divides $N$), denoted by $\mathcal{C}_{k,q}$ ($q=1,2, \dots, \tfrac{N}{p}$), and particles interact among each other within each batch. Then one reshuffles and continues. This shuffle-interact algorithm is formulated in algorithm RBM-1 below, and clearly it only has complexity of $O(n_T pN)$ for fixed time $T$.

% As described in RBM-1 below, each iteration consists of two steps: (1) Randomly shuffling and dividing the particles into $\tfrac{N}{p}$ batches; (2) evolving with interactions only turned on within the batches. Clearly, such constructed method only has complexity of $O(n_T pN)$ for fixed time $T$.

\begin{algorithm}[H]
\caption{RBM-1}\label{al:RBM1}
\begin{algorithmic}[1] % The [1] option enables line numbering
\For{$k \in \{0, \ldots, \frac{T}{\kappa}-1\}$}
    \State Divide $\{1, 2, \ldots, N\}$ into $\tfrac{N}{p}$ batches randomly.
    \For{each batch $\mathcal{C}_{k,q}$ ($q = 1,2,\dots, \tfrac{N}{p}$)} 
        \State Update $X^i$ (where $i \in \mathcal{C}_{k,q}$) by solving the following SDE with $t \in [t_{k}, t_{k+1}]$:
        \begin{equation*}
        dX^i = -\nabla V(X^i)dt + \frac{1}{p-1} \sum_{\substack{j \in \mathcal{C}_{k,q},  j \neq i}} K(X^j - X^i)dt + \sigma dB^i.
        \end{equation*}
    \EndFor
\EndFor
\end{algorithmic}
\end{algorithm}

% In the above double-loop Algorithm \ref{al:RBM1}, at each $t_k$ ($k = 0,1,2,\dots$), one randomly divides the $N$ particles into $\tfrac{N}{p}$ groups. Then in the inner loop, at $q$-th time, the $p$ randomly picked particles $X^i$ ($i \in C_{k,q}$) evolve with the interaction force only turned on within $C_{k,q}$ while other $N-p$ particles $X^i$ ($i \notin C_{k,q}$) stays still. At $(q+1)$-th time, again one randomly picks $p$ particles from those never picked in $\cup_{q' = 1}^{q-1} C_{k,q'}$ and similarly runs the dynamics. Consequently, $\cup_{q=1}^{\tfrac{N}{p}} C_{k,q}= \{1,2\dots,N\}$ and $C_{k,q_1} \cap C_{k,q_2} = \empty$ for any $q_1 \neq q_2$. Hence, this kind of choice of random batch can be viewed as sampling without replacement and we call the algorithm RBM-1 the Random Batch Method without replacement.

An alternative and experimentally effective approach is the pick-interact algorithm (RBM-r below). As discussed at the beginning, the idea is to apply the kinetic Monte Carlo to interacting particle system \eqref{eq:IPSintro0}. Namely, at each time grid $t_k$, one randomly picks a batch $\mathcal{C}_k$, and only the particles in the picked batch $\mathcal{C}_k$ interact among each other within the batch for a short time $\kappa$. Then one repeats this operation for $t_{k+1}, t_{k+2}, \dots$. One natural question on this application of KMC is, why does it work? (i.e. why can the proposed dynamics simulates / approximates the original interacting particle system \eqref{eq:IPSintro0}?) To answer this question, it is crucial to focus on the following during the algorithm construction:
\begin{enumerate}
    \item[(1).] Careful counting of the effective time in RBM-r: 
    compared with time $t$ in RBM-1 or \eqref{eq:IPSintro0}, the effect time in RBM-r should be $\tfrac{N}{p}t$ in the sense of the $L^2$ weak law of large number (see Theorem \ref{eq:thm:mainthm} below, also see more detailed discussions in Remark \ref{rmk:pseudotime} below).
    \item[(2).] Independent selection of the random batches $\mathcal{C}_0,\mathcal{C}_1,\dots$. This kind of independent selection of random batches can be viewed as sampling with replacement and thus this proposed alternative algorithm is called the Random Batch Method with replacement (RBM-r) \cite{jin2020random}.
\end{enumerate}
The above RBM-r algorithm (pick-interact algorithm or KMC for \eqref{eq:IPSintro0}) is displaced below:
\begin{algorithm}[H]
\caption{RBM-r}\label{al:rbmr}
\begin{algorithmic}[1]
\For{$k = 0$ \textbf{to} $\tfrac{N}{p}\frac{T}{\kappa}-1$}
    \State Pick a set $\mathcal{C}_k$ of size $p$ randomly.
    \State Update $X^i$ (where $i \in \mathcal{C}_k$) by solving the following with $s \in [k\kappa, (k+1)\kappa]$:
    \[
    dX^i = -\nabla V(X^i) ds + \frac{1}{p-1} \sum_{\substack{j \in \mathcal{C}_k,  j \neq i}} K(X^j - X^i) ds + \sigma dB^i.
    \]
\EndFor
\end{algorithmic}
\end{algorithm}

RBM-r and RBM-1 share similar implementation details, convergence properties, and the same computational complexity (clearly, the complexity of RBM-r is $O(n_Tp^2\tfrac{N}{p}) = O(n_T p N)$). To give a clearer picture of RBM-r is indeed a variant of RBM-1, we can reformulate RBM-r into the following double-loop form: for $k = 0,\dots, \frac{T}{\kappa} - 1$, at each $t_k = k\kappa$, choose $\tfrac{N}{p}$ independent batches $\mathcal{C}_{k,q}$ ($1 \leq q \leq \tfrac{N}{p}$) of size $p$. Note that $\mathcal{C}_{k,q}$ for different $q$ may have overlaps. Then, for each $q = 1, \dots, \tfrac{N}{p}$, update $X^i$ (for $i \in \mathcal{C}_{k,q}$) by solving the following SDE for time $\kappa$:

\begin{equation}\label{eq:RBMr'}
    dY^i = -\nabla V(Y^i)dt + \frac{1}{p-1} \sum_{\substack{j \in \mathcal{C}_{k,q},  j \neq i}} K(Y^j - Y^i)dt + \sigma dB^i,\quad Y^i(0) = X^i,
\end{equation}
and set $X^i \leftarrow Y^i(\kappa)$. Clearly, \eqref{eq:RBMr'} differs from RBM-1 only in the choice of random batches. In detail, in RBM-1, one requires the non-intersected union $\sqcup_{q=1}^{\tfrac{N}{p}} \mathcal{C}_{k,q}= \{1,2\dots,N\}$, but in \eqref{eq:RBMr'}, the batches are independent and thus may have overlaps. 
Also, it is straightforward to give a practical comparison between RBM-r and RBM-1 outlined above. On one hand, RBM-r is easier to implement than RBM-1, especially for the fully vectorized version. For instance, when implementing RBM-1 in $\mathbb{R}^d$, in each time interval, after shuffling the $N$ particles into $p$ groups, the key and most expansive step is to compute the drift (in particular, the value of the interaction kernel). In this step, one needs to reshape the original $N \times d$ tensor into a higher dimensional one according to the shuffling, and then calculate the (vectorized) interaction kernel using the broadcasting mechanisms, which is supported by most programming languages. On the other hand, the parallelization efficiency of RBM-r is not as good as that of RBM-1, since in the RBM-1 one can simulate $N$ particles simultaneously and run the $N$-particle system over time $T$, while the RBM-r only allows to run $p$ particles at the same time since the batch is chosen at random.

Moreover, since the interaction force $\frac{1}{p-1} \sum_{\substack{j \in \mathcal{C}_{k},  j \neq i}} K(X^j - X^i)$ in RBM-r has the consistency property as we discussed in \eqref{eq:unbiased}, intuitively RBM-r would also have convergence.
To be more detailed, at first glance, after only a few such operations, the resulted dynamics would differ a lot from the original dynamics \eqref{eq:IPSintro0}. However, for $\kappa \rightarrow 0$, the resulted dynamics would be close to the original one due to the $L^2$ weak law of large number.
Numerical experiments in Section 5 of \cite{jin2020random} confirmed that RBM-r at time $\tfrac{N}{p}t$ 
gives acceptable approximation for the distribution of \eqref{eq:IPSintro0} at time $t$, while there is no convergence for RBM-r yet. The contribution of this paper is to address this question and we prove that RBM-r at time $\tfrac{N}{p}t$ converges to the interacting particle system \eqref{eq:IPSintro0} at time $t$ in the Wasserstein-2 distance as $\kappa \rightarrow 0$, with an explicit convergence rate. (See Theorem \ref{eq:thm:mainthm} below for more details.)

% while the rigorous analysis for the convergence of RBM-r is still open 

Since proposed in 2020, the Random Batch Method has been applied to various practical algorithms. For instance, in \cite{jin2021random, liang2022superscalability}, researchers combined the RBM and the Ewald algorithm and proposed the Random Batch Ewald method for the molecular dynamics. Other recent extensions of RBM include its application to sampling \cite{li2020random}, interacting quantum particle systems \cite{golse2019random, jin2020randomquantum}, second-order interacting particle systems \cite{jin2022random}, multi-species stochastic interacting
particle systems \cite{daus2022random}, the Poisson-Nernst-Planck (PNP) equation \cite{li2022some}, the homogeneous Landau equation \cite{carrillo2021random}, to name a few. Moreover, the proposed Random Batch Method has yielded fruitful theoretical results. In \cite{jin2020random}, the authors proved that the Random Batch Method without replacement (RBM-1) converges to the original interacting particle system \eqref{eq:IPSintro0} as $\kappa \rightarrow 0$, with a convergence rate under Wasserstein-2 distance of $O(\sqrt{\frac{\kappa}{p-1} + \kappa^2})$. In \cite{jin2021convergence}, an $O(\sqrt{\kappa})$ strong error and an $O(\kappa)$ weak error of RBM for interacting particle system with weights and multispecies (i.e. the interaction force is of the form $\frac{1}{N-1} \sum_{j: j \neq i} m_j K_{i j}\left(X^i, X^j\right)$) were established. The geometric ergodicity of the RBM dynamics as well as some convergence of numerical discretization for RBM of dynamics has also been well studied via some specific coupling methods \cite{jin2023ergodicity, ye2024error}. Recently, a sharp error bound for the numerical discretization of RBM has been established \cite{huang2024mean}, motivated by a series of advanced techniques in data science and machine learning communities \cite{li2022sharp, mou2022improved, jabin2018quantitative}. Furthermore, the mean-field limit (the $N \rightarrow \infty$ region) of the Random Batch Method has also been studied in recent years \cite{jin2022mean, ye2024error}. However, for the related method RBM-r, which has demonstrated outstanding performance in various tasks \cite{jin2020random}, rigorous convergence analysis remains open.
The main contribution of this paper is to fill in this gap.

As mentioned above, the proposed RBM-r can be regarded as the kinetic Monte Carlo applied to the pairwise interacting particle system \eqref{eq:IPSintro0}. Originating in the 1970s to study the Ising model \cite{bortz1975new}, KMC has become a classical algorithm widely used across multiple fields, especially in statistical physics. The kinetic Monte Carlo differs from the simple Metropolis Monte Carlo \cite{metropolis1949monte}, in the sense that it evolves the system dynamically from state to state, which overcomes some shortcomings of Metropolis's method including low acceptance rate and slow convergence. This probabilistic selection of move-forward events in KMC is exactly repeatedly, independently choosing the random batch in RBM-r. Moreover, as discussed in \cite{voter2007introduction}, although people focus more on the equilibrium state in most research on KMC, the KMC algorithm can intuitively provide a good simulation of a dynamical system. Unfortunately, to the best of our knowledge, the rigorous convergence of KMC to such a dynamical systems is still unclear so far. Therefore, establishing suitable robust methods for the convergence of KMC (or RBM-r in our case) is meaningful and necessary.

\subsection{A preview of the main theorem and key techniques}

In the following, we will give a brief summary of our main result - convergence of RBM-r in Wasserstein-2 distance, as well as some key techniques applied during our proof. The dynamic of RBM-r can be written as follows: for $k = 0,1,2,\dots$, $t_k = k\kappa$, and $1 \leq i \leq N$,
\begin{equation}\label{eq:RBMintro}
\begin{aligned}
    &d\tilde{X}^i = -\nabla V(\tilde{X}^i) dt + \frac{1}{p-1}\sum_{j\in \mathcal{C}_k,j\neq i}K\left(\tilde{X}^j - \tilde{X}^i \right) dt + \sigma d\tilde{B}^i,\quad i \in \mathcal{C}_k,\quad t \in [t_k,t_{k+1}),\\
    &d\tilde{X}^i = 0,\quad i \notin \mathcal{C}_k,\quad t \in [t_k,t_{k+1}),
\end{aligned}
\end{equation}
where $\mathcal{C}_k$ is the random batch we pick at $t_k$. More precisely, each possible configuration of $\mathcal{C}_k$ (i.e. a $p$-element subset of $\{1,\dots,N\}$) is sampled with equal probability. Recall the definition of $X^i$ at \eqref{eq:IPSintro0}.  Note that both the two systems $(X^i)_{1\leq i \leq N}$ and $(\tilde{X}^i)_{1 \leq i \leq N}$ are exchangeable, meaning that $Y^1,\dots, Y^N$ has the same distribution, where $(Y^i)_{i=1}^N$ represents for $(X^i)_{1\leq i \leq N}$ or $(\tilde{X}^i)_{1 \leq i \leq N}$. For any $t \geq 0$, denote $\rho_t^{(1)}$, $\tilde{\rho}_t^{(1)}$ the first marginal distributions for solutions of \eqref{eq:IPSintro0}, \eqref{eq:RBMintro} at time $t$, respectively. We analyze the distance between $\rho_t^{(1)}$ and $\tilde{\rho}_t^{(1)}$, and in our main result we consider the Wasserstein-2 ($W_2$) distance. For any two probability measure $\mu$, $\nu$, $W_2(\mu,\nu)$ is defined by
\begin{equation*}
    W_{2}(\mu, \nu):=\left(\inf _{\gamma \in \Pi(\mu, v)} \int_{\mathbb{R}^{d} \times \mathbb{R}^{d}}|x-y|^{2} d \gamma\right)^{1 / 2}.
\end{equation*}
Our main result can be roughly stated as follows (the rigorous statement of this result is Theorem \ref{eq:thm:mainthm} below):

\begin{theorem*}
Consider the interacting particle system \eqref{eq:IPSintro0} and the RBM-r algorithm \eqref{eq:RBMintro}. Assume the following conditions:
\begin{itemize}
    \item $V(\cdot)$ is $\lambda$-strongly convex and $\nabla V(\cdot)$ grows at most polynomially,
    \item $K(\cdot)$ is bounded and $L$-Lipschitz.
    \item $\lambda > 2L$
    \item $X_0^i = \tilde{X}_0^i$ for all $1 \leq i \leq N$.
\end{itemize}
Then, fixing $T>0$, for any small time step $\kappa$, the Wasserstein-2 distance between first marginal distributions of \eqref{eq:IPSintro0} and \eqref{eq:RBMintro} is bounded by
\begin{equation}\label{eq:mainthmintro}
    \sup_{0\leq t \leq T}W_2(\tilde{\rho}_{\bar{t}}^{(1)},\rho_t^{(1)}) \leq C \left(1 + T^{\tfrac{1}{4}} \right)\kappa^{\frac{1}{4}},
\end{equation}
where $C$ is a positive constant only depending on $\lambda$, $L$, $T$, and $\bar{t}:= \tfrac{N}{p}t$ denotes the pseudo-time for any physical time $t \in [0,T]$.
\end{theorem*}

Note that the convergence rate ($O(\kappa^{\frac{1}{4}})$) is weaker than that of RBM-1 ($O(\kappa^{\frac{1}{2}})$). Technically, this is because we use the random time change in our construction of coupling (see below and in Section \ref{sec:mainresult}). Consequently, it is unavoidable to make use of the H\"older's continuity property of the differential equation (see Lemma \ref{lmm:hatX_continuity} below), leading to a weaker result. However, it is possible to obtain an improved convergence with rate $O(\kappa^{\frac{1}{2}})$ when there is no diffusion in the system (i.e. $\sigma \equiv 0$), see Theorem \ref{coro:sigma0} below for a more detailed discussion.

Next, we give an overview of our proof as well as some key techniques applied in our analysis.
From the exchangeability, it is a natural idea to prove \eqref{eq:mainthmintro} via some suitable coupling method, namely, finding some realizations $X^i$, $\tilde{X}^i$ of \eqref{eq:IPSintro0}, \eqref{eq:RBMintro}, coupled in the specific way (for instance, $X^i$, $\tilde{X}^i$ driven by the same Brownian motion) such that
\begin{equation*}
    \sup_{0\leq t \leq T}\sqrt{\mathbb{E}\left|\tilde{X}^{(1)}_{\bar{t}} - X^{(1)}_t \right|^2} \lesssim \kappa^{\frac{1}{4}}.
\end{equation*}
A direct construction of the paired dynamics $(X,\tilde{X})$ is not straightforward. Our key insight here is to introduce an intermediate dynamics IPS' in \eqref{eq:IPS'intro} below, which consists of $N$ copies of IPS after random time change. And then we make use of the diagonal principle during the analysis. In details, we define the following dynamics triple $\left(\tilde{X}, \hat{X}, X\right)$  with the same initial state
\begin{equation*}
    \tilde{X}^\ell(0)= \hat{X}^{\ell,i}(0)= X^\ell(0), \quad \forall 1\leq l,i \leq N.
\end{equation*} 
 \textbf{RBM-r:} for $1\leq i \leq N$, $t \in[t_k,t_{k+1})$, $k = 0,1,2,\dots$,
\begin{equation}\label{eq:RBM-rintro}
\begin{aligned}
    &d\tilde{X}^{i} = -\nabla V(\tilde{X}^{i}) dt + \frac{1}{p-1}\sum_{j\in \mathcal{C}_k,j\neq i}K\left(\tilde{X}^j - \tilde{X}^i \right) dt + \sigma d\tilde{B}^{i},\quad i \in \mathcal{C}_k,\\
    &d\tilde{X}^i = 0,\quad i \notin \mathcal{C}_k,
\end{aligned}
\end{equation}
\textbf{IPS':} for $1\leq i \leq N$, $1 \leq \ell \leq N$, $t \in[t_k,t_{k+1})$, $k = 0,1,2,\dots$,
\begin{equation}\label{eq:IPS'intro}
\begin{aligned}
    &d\hat{X}^{\ell,i} = -\nabla V(\hat{X}^{\ell,i}) dt + \frac{1}{N-1}\sum_{j\neq l}K\left(\hat{X}^{j,i} - \hat{X}^{\ell,i} \right) dt + \sigma d\hat{B}^{\ell,i},\quad i \in \mathcal{C}_k,\\
    &d\hat{X}^{\ell,i} = 0,\quad i \notin \mathcal{C}_k,
\end{aligned}
\end{equation}
\textbf{IPS:} for $1\leq \ell \leq N$, $t \in[t_k,t_{k+1})$, $k = 0,1,2,\dots$,
\begin{equation}\label{eq:IPSintro}
dX^\ell = -\nabla V(X^\ell) dt + \frac{1}{N-1}\sum_{j\neq l}K\left(X^j - X^\ell \right) dt + \sigma dB^\ell,
\end{equation}
Above, at each $k$-th interval, RBM-r \eqref{eq:RBM-rintro} and IPS' \eqref{eq:IPS'intro} share the same random batch $\mathcal{C}_k$. Also, the Brownian motions above are coupled such that for any $1 \leq \ell \leq N$, $1 \leq i \leq N$, $n = 0,1,2,\dots$, $\Delta t \in [0,\kappa)$,
\begin{equation}\label{eq:coupleBMintro}
\begin{aligned}
    &\hat{B}^{\ell,i}(\tau_n^{i} \kappa + \Delta t) = B^\ell(n\kappa+\Delta t),\\
    &\tilde{B}^i(\tau_n^{i} \kappa + \Delta t) = B^{i}(n\kappa + \Delta t),
\end{aligned}
\end{equation}
where for each fixed $i$, the stopping time $\tau_n^i$ is defined to be the $n$-th time the particle index $i$ is chosen into the batch.
\begin{equation}\label{eq:defmniintro}
    \tau_n^i := \inf\left\{K: \sum_{k=0}^K \textbf{1}_{\{i \in \mathcal{C}_k \}} > n \right\}.
\end{equation}
From the coupling of the Brownian motion in \eqref{eq:coupleBMintro}, it is easy to see that the index $\ell$ above represents the $\ell$-th particle in system, while the index $i$ represents the $i$-th copy of the particle system after the random time change.

Also note that from the way we couple the Brownian motion in \eqref{eq:coupleBMintro}, it is clear that $\tilde{X}^i$ in RMB-r \eqref{eq:RBM-rintro} and diagonal elements $\hat{X}^{i,i}$ in IPS' \eqref{eq:IPS'intro} are synchronously coupled in that
\begin{equation*}
    \tilde{B}^i \equiv \hat{B}^{i,i}
\end{equation*}
for each $i$. Furthermore, from the definition of the RBM-r \eqref{eq:RBM-rintro} and the way we couple the Brownian motions, we see that IPS' \eqref{eq:IPS'intro} defines $N$ pairwise independent interacting particle systems. Indeed, for each fixed $i$, all particles $\hat{X}^{\ell,i}$ ($1\leq \ell \leq N$) are moving forward (with the same interaction force) or staying still simultaneously according to whether the event $\{i \in \mathcal{C}_k\}$ happens. Consequently, we find that each $N$-body system $\left(\hat{X}^{\ell,i}\right)_{\ell=1}^N$ is a random time change from $(X^\ell)_{\ell=1}^N$ in IPS \eqref{eq:IPSintro}, induced by index $i$. This means
\begin{equation}\label{eq:comparetwoIPSintro}
    \hat{X}^{\ell,i}(\tau_n^i \kappa + \Delta t) = X^\ell(n\kappa + \Delta t), \quad \forall \, \Delta t\in [0,\kappa), \quad n = 0,1,2,\dots,\quad 1\leq \ell \leq N.
\end{equation}
This random time change relation is crucial in the subsequent analysis. Another key observation is that our construction preserves the exchangeability property of $Z^i := \hat{X}^{i,i} - \tilde{X}^i$ when considering the diagonal element $\hat{X}^{i,i}$. Therefore, in the practical calculation, we are able to make use of these diagonal elements of IPS' as an intermediate bridge for the analysis.

After constructing the coupling above, recalling that our goal is to estimate the distance between $\tilde{X}^i(\tfrac{N}{p}t)$ and $X^i(t)$ for $0\leq t \leq T$, it is natural to estimate the distance between IPS and RBM-r by the following two main steps. In the following, we would consider fixed $t \in [0,T]$ (also recalling the notation of the pseudo-time $\bar{t} = \tfrac{N}{p}t$) and assume $\kappa$ divides $t$ for simplicity.
\begin{itemize}
    \item Compare IPS and IPS': estimate $\mathbb{E}\left|\hat{X}^{i,i}(\bar{t})  - X^i(t)\right|^2$.
    \item Compare IPS' and RBM-r: estimate $\mathbb{E}\left|\tilde{X}^i(\bar{t}) - \hat{X}^{i,i}(\bar{t}) \right|^2$.
\end{itemize}

To compare IPS and IPS', we make use of random time change relation \eqref{eq:comparetwoIPSintro}, the H\"older's continuity property proved in Lemma \ref{lmm:hatX_continuity} below and the $L^2$ weak law of large number. After careful calculation, we will establish that
\begin{equation*}
    \mathbb{E}\left|\hat{X}^{i,i}(\bar{t})  - X^i(t)\right|^2 = \mathbb{E}\left|\hat{X}^{i,i}(\tfrac{N}{p}t)  - \hat{X}^{i,i}(\kappa m^i_{n_t})\right|^2 \lesssim \sqrt{\kappa},
\end{equation*}
where $n_t := t/\kappa$.

Comparing IPS' and RBM-r is more complicated. Define
\begin{equation}\label{eq:defZintro}
    Z^i = \tilde{X}^i - \hat{X}^{i,i} \quad \text{for} \quad 1\leq i \leq N.
\end{equation}
The strategy is to decompose $\frac{d}{dt}\mathbb{E}|Z^i(t)|^2$ into the following three parts:
\begin{equation}\label{eq:splitintro}
\begin{aligned}
    \frac{d}{ds}\mathbb{E}\left|Z^i(s) \right|^2 &= -\mathbb{E} \left[Z^i(s) \cdot \left(\nabla V\left(\tilde{X}^i(s) \right) - \nabla V\left(\hat{X}^{i,i}(s) \right)\right)\right]\\
    &\quad+ \frac{1}{N-1}\sum_{j \neq i}\mathbb{E}\left[Z^i(s) \cdot \left( K \left(\tilde{X}^j(s) - \tilde{X}^i(s) \right) - K\left( \hat{X}^{j,i}(s) - \hat{X}^{i,i}(s)\right)\right)\right]\\
    &\quad+ \mathbb{E}\left[Z^i(s) \cdot \mathcal{X}_k^i(\tilde{X}(s))\right],
\end{aligned}
\end{equation}
where
\begin{equation*}
    \mathcal{X}_k^i\left(\tilde{X}(s) \right):= \frac{1}{p-1}\sum_{j\in \mathcal{C}_k,j\neq i}K\left(\tilde{X}^j(s) - \tilde{X}^i(s) \right) - \frac{1}{N-1}\sum_{j\neq i}K\left(\tilde{X}^j(s) - \tilde{X}^i(s) \right).
\end{equation*}
Employing the strong convexity of $V(\cdot)$, it is straghtforward to estimate the first term on the right-hand side of \eqref{eq:splitintro} above. We handle the second term using the Lipschitz assumption for the interaction kernel $K(\cdot)$, exchangeability of $Z$, and the $L^2$ weak law of large numbers. See more details in Lemma \ref{lmm:exchange} below. The third term represents the error induced by the random batch. We will estimate it using the consistency of random batch $\mathcal{C}_k$, the Lipschitz assumption for $K(\cdot)$, and the continuity property in Lemma \ref{lmm:hatX_continuity}. See more details in Lemma \ref{eq:deltakai} below.

The rest of this paper is organized as follows. 
Before proving our main result, in Section \ref{sec:lemmas}, we establish some crucial technical lemmas used in the proof. In Section \ref{sec:mainresult}, we present and prove our main result - convergence of RBM-r in Wasserstein-2 distance, along with some remarks and corollaries. In Section \ref{sec:concludion}, we give a conclusion and some discussions on possible improvement of the result, including potential applications of our analysis tools to other algorithms like kinetic Monte Carlo for different physical models, and potential methodologies to improve the convergence rate. Omitted proofs of some lemmas in Section \ref{sec:lemmas} are provided in Appendix \ref{sec:append}.

% In Section \ref{sec:mainresult}, we present and prove our main result - convergence of RBM-r in Wasserstein-2 distance. Specifically, we present our assumptions and state our main theorem, Theorem \ref{eq:thm:mainthm}, along with some remarks and corollaries in Section \ref{sec:assresult}, and then prove Theorem \ref{eq:thm:mainthm} in Section \ref{sec:proof} assuming some technical lemmas. In Section \ref{sec:lemmas}, establish these technical lemmas. In Section \ref{sec:concludion}, we give a conclusion and some discussions on possible improvement of the result, including potential applications of our analysis tools to other algorithms like kinetic Monte Carlo for different physical models. Omitted proofs of some lemmas in Section \ref{sec:lemmas} are provided in Appendix \ref{sec:append}.

\section{Lemmas used in the proof of Theorem \ref{eq:thm:mainthm}}\label{sec:lemmas}

Before proving the main result (stated in \eqref{eq:mainthmintro} above), we first establish some auxiliary result in this section. Lemma \ref{lmm:moment} 
provides uniform-in-time moment control for the stochastic dynamics considered in \eqref{eq:RBM-rintro} - \eqref{eq:IPSintro}, and is used for multiple times during the proof of our main result (Theorem \ref{eq:thm:mainthm} below), Lemma \ref{lmm:hatX_continuity} and Lemma \ref{eq:deltakai}. Lemma \ref{lmm:hatX_continuity} demonstrates the H\"older's continuity property of differential equations, which is applied in STEP 2 of the main proof (see Section \ref{sec:mainresult} below) and in proofs of Lemmas \ref{lmm:exchange}, \ref{eq:deltakai} below. Lemma \ref{eq:L2LLN} proves a $L^2$ weak law of large numbers and is used in STEP 2 of the main proof. Lemma \ref{lmm:exchange} establishes the exchangeability of the system $(Z^i)_{1 \leq i \leq N}$ defined in \eqref{eq:defZintro} and provides upper bounds for the second term on the right-hand side of \eqref{eq:splitintro} in STEP 3 of the main proof. Lemma \ref{eq:deltakai} controls the error brought by the random batch, and provides upper bounds for the third term on the right-hand side of \eqref{eq:splitintro}. All the results in this section are based on the conditions assumed in Section \ref{sec:intro} before \eqref{eq:mainthmintro} (see also Assumption \ref{ass} below).

To begin with, via standard It\^o's calculation, we are able to prove the uniform-in-time moment control in the following lemma. Note that the positive constant $C_q$ in Lemma \ref{lmm:moment} below may also depend implicitly on $\lambda$, $L$, $\|K\|_{\infty}$, $\sigma$ and $d$, and we denote it by $C_q$ only to simplify the notation. For the readers' convenience, we move its proof to Appendix \ref{sec:lmmmoment}.

\begin{lemma}\label{lmm:moment}[moment control]
Consider $(\tilde{X}^i(t))_i$ defined in RBM-r \eqref{eq:RBM-rintro}, $(\hat{X}^{l,i}(t))_{l,i}$ defined in IPS' \eqref{eq:IPS'intro} and  $(X^i(t))_i$ defined in IPS \eqref{eq:IPSintro}. Suppose Assumption \ref{ass} holds. Then, for any $q \geq 2$, there exists a positive constant $C_q$ independent of $N$ and $p$ such that 
$$
\sup_{1\leq i \leq N}\sup _{t \geq 0}\mathbb{E}\left|\tilde{X}^i(t)\right|^q \leq C_q ,\quad \sup_{1\leq i,\ell \leq N}\sup _{t \geq 0}\mathbb{E}\left|\hat{X}^{\ell,i}(t)\right|^q \leq C_q, \quad
\sup_{1\leq i \leq N}\sup _{t \geq 0}\mathbb{E}\left|X^{i}(t)\right|^q \leq C_q.
$$
% Besides, for any $k>0$ and $q \geq 2$,
% $$
% \sup _{t \in\left[t_{k-1}, t_k\right)}\left|\mathbb{E}\left[\left|\tilde{X}^1(t)\right|^q \mid \mathcal{F}_{k-1}\right]\right| \leq C_1\left|\tilde{X}^1\left(t_{k-1}\right)\right|^q+C_2 .
% $$
% holds almost surely, where $\mathcal{F}_{k-1} := \sigma \left(\tilde{X}^1(0); B^1_s, s\leq t_{k-1}; C_j, j\leq k-1 \right)$.
% Moreover, for $t \in\left[t_{k-1}, t_k\right)$, there exist a constant $C>0$ independent of $C_k, k, N$ and an index $q>0$ such almost surely that
% $$
% \left|\mathbb{E}\left[\tilde{X}^1(t)-\tilde{X}^1\left(t_{k-1}\right) \mid \mathcal{F}_{k-1}\right]\right| \leq C\left(1+\left|\tilde{X}^1\left(t_{k-1}\right)\right|^q\right) \kappa,
% $$
% and in $L^2(\Omega)$ that
% $$
% \mathbb{E}\left|\mathbb{E}\left[\left|\tilde{X}^1(t)-\tilde{X}^1\left(t_{k-1}\right)\right|^2 \mid \mathcal{F}_{k-1}\right]\right|^2 \leq C \kappa^2 .
% $$
\end{lemma}

In the next lemma, we prove the H\"older's continuity property of differential equations. This result is applied in STEP 2 of the main proof and in proofs of Lemma \ref{lmm:exchange}, \ref{eq:deltakai} below. The proof is based on basic inequalities and moment control in Lemma \ref{lmm:moment}. For readers' convenience, we move the detailed proof to Appendix \ref{sec:lmmconti}.

\begin{lemma}\label{lmm:hatX_continuity}[H\"older's continuity]
Consider the processes $\hat{X}^{\ell,i}$ ( $1 \leq i,\ell \leq N$) defined in IPS' \eqref{eq:IPS'intro}, $X^l$ ($1 \leq l \leq N$) in IPS \eqref{eq:IPSintro}, $\tilde{X}^i$ ($1\leq i \leq N$) defined in RBM-r \eqref{eq:RBM-rintro}, and $Z^i = \tilde{X}^i - \hat{X}^{i,i}$. Under Assumption \ref{ass}, for any $t_1$, $t_2 > 0$, there exists a positive constant $C$ independent of $t_1$, $t_2$, $\kappa$, $i$, $\ell$, $N$, $p$, $\sigma$ such that the following holds:
\begin{equation}\label{eq:timecon2}
    \mathbb{E}\left|X^\ell(t_2) - X^\ell(t_1) \right|^2 \leq3\sigma^2|t_2 - t_1|+C|t_2 - t_1|^2,
\end{equation}
and
\begin{equation}
    \mathbb{E}\left|Z^i(t_1) - Z^i(t_2)\right|^2 \leq C|t_2-t_1|^2.
\end{equation}
\end{lemma}

The following Lemma proves a $L^2$ weak law of large number. In detail, we prove by studying the exact probability distribution of the random variable defined by (recall the definition of the stopping time $\tau_n^i$ defined in \eqref{eq:defmniintro}):
$$\tilde{\tau}_k^i := \tau_k^i - \tau_{k-1}^i.$$
This result is used during STEP 2 of the main proof. We move the detailed calculation to Appendix \ref{sec:lmmlln}.

\begin{lemma}[$L^2$ weak law of large number]\label{eq:L2LLN}
For any fixed $i$, recall the definition of the stopping time $\tau_n^i$ defined in \eqref{eq:defmniintro}:
\begin{equation*}
    \tau_n^i := \inf\left\{K: \sum_{k=0}^K \textbf{1}_{\{i \in \mathcal{C}_k \}} > n \right\}.
\end{equation*}
Then fixing any $t>0$, for any $\kappa>0$ such that $\kappa$ divides $t$ and $n_t = t / \kappa$, $\tau^i_{n_t}$ has mean $\tfrac{N}{p}n_t$ and variance $((\tfrac{N}{p})^2 - \tfrac{N}{p})n_t$, and the following $L^2$ weak law of large number holds:
\begin{equation*}
    \mathbb{E}\left|\kappa \tau_{n_t}^i - \tfrac{N}{p}t \right|^2  = t^2\mathbb{E}\left|n_t^{-1}\tau_{n_t}^i - \tfrac{N}{p} \right|^2=\left(\tfrac{N^2}{p^2} - \tfrac{N}{p} \right)\kappa^2\left(\kappa^{-1}t + 1 \right) \rightarrow 0 \quad as \quad \kappa \rightarrow 0.
\end{equation*}
\end{lemma}

Lemma \ref{lmm:exchange} below establishes the exchangeability of the system $(Z^i)_{1 \leq i \leq N}$ defined in \eqref{eq:defZintro}. It also provides upper bounds for the second term on the right-hand side of \eqref{eq:splitintro}. The proof relies on the Lipschitz assumption for $K(\cdot)$ and Lemma \ref{eq:L2LLN}, which provides detailed calculation on properties of the stopping time  $\tau_n^i$ defined in \eqref{eq:defmniintro}. 

\begin{lemma}\label{lmm:exchange}
    Recall the definition of $\tilde{X}^i(t)$ ($1 \leq i \leq N$) in RBM-r \eqref{eq:RBM-rintro} and $\hat{X}^{\ell,i}(t)$ ($1\leq \ell,i\leq N$) in IPS' \eqref{eq:IPS'intro}. Then the processes
    \begin{equation*}
        Z^i(t) = \tilde{X}^i(t) - \hat{X}^{i,i}(t),\quad t\geq 0,\quad 1 \leq i \leq N
    \end{equation*}
    are exchangeable, which means $Z_1(t)$, $Z_2(t)$, $\dots$, $Z_N(t)$ have the same distribution for any $t$. Moreover, \begin{enumerate}
    \item For $\delta > 0$ to be determined,
    \begin{multline}\label{eq:2nd1inlmm}
    \frac{1}{N-1}\sum_{j \neq i}\mathbb{E}\left[Z^i(t) \cdot \left( K \left(\tilde{X}^j(t) - \tilde{X}^i(t) \right) - K\left( \hat{X}^{j,i}(t) - \hat{X}^{i,i}(t)\right)\right)\right]\\
    \leq 2L\mathbb{E}\left|Z^i(t) \right|^2 + \delta \mathbb{E}\left|Z^i(t) \right|^2 + \frac{1}{4\delta}\frac{1}{N-1}\sum_{j \neq i} \mathbb{E}\left|\hat{X}^{j,j}(t) - \hat{X}^{j,i}(t) \right|^2, \quad \forall t \geq 0.
    \end{multline}
    \item For fixed $T>0$, for small $\kappa$, we have
    \begin{equation}\label{eq:2nd2inlmm}
     \mathbb{E}\left|\hat{X}^{j,j}(s) - \hat{X}^{j,i}(s) \right|^2 \leq C\sqrt{T\kappa}, \forall s \in [0,T],
    \end{equation}
    where $C$ is a positive constant independent of $N$, $p$, $i$, $j$, $T$, $t$ $\kappa$.
\end{enumerate}
\end{lemma}
\begin{proof}
By definition, the index of the $N$ particles are i.i.d. chosen (or not chosen) into the random batch $\mathcal{C}_k$ at $k$-th interval. Moreover, $d\tilde{X}^{i} = d\hat{X}^{i,i} = 0$ for $t \in [t_k,t_{k+1})$ if $i \notin \mathcal{C}_k$, and for each $i$, $\hat{X}^{i,i}$ is a result of change of measure only according to the behavior of index $i$. Therefore, the $N$ indexes $1,2,\cdots, N$ play the equal roles when constructing $(Z^i)_{i=1}^N$ and thus $(Z^i)_{i=1}^N$ are exchangeable.

Now we prove \eqref{eq:2nd1inlmm}. In fact, using the assumed Lipschitz condition of $K(\cdot)$ and the exchangeability of $Z^i$, we have
\begin{equation*}
    \begin{aligned}
        &\quad\frac{1}{N-1}\sum_{j \neq i}\mathbb{E}\left[Z^i(t) \cdot \left( K \left(\tilde{X}^j(t) - \tilde{X}^i(t) \right) - K\left( \hat{X}^{j,i}(t) - \hat{X}^{i,i}(t)\right)\right)\right]\\
        &\leq L \frac{1}{N-1}\sum_{j \neq i} \mathbb{E}\left[\left|Z^i(t)\right|\cdot \left|Z^j(t) - Z^i(t) + \left(\hat{X}^{j,j}(t) - \hat{X}^{j,i}(t) \right) \right| \right]\\
        &\leq  L\frac{1}{N-1}\sum_{j\neq i}\mathbb{E}\left[\left|Z^i(t) \right|\left|Z^j(t) \right|\right] + L\mathbb{E}\left|Z^i(t) \right|^2
        + \frac{1}{N-1}\sum_{j \neq i} \mathbb{E}\left[\left| Z^i(t)\right|\left|(\hat{X}^{j,j}(t) - \hat{X}^{j,i}(t) \right|\right]\\
        &\leq 2L\mathbb{E}\left|Z^i(t) \right|^2 + \delta \mathbb{E}\left|Z^i(t) \right|^2 + \frac{1}{4\delta}\frac{1}{N-1}\sum_{j \neq i} \mathbb{E}\left|(\hat{X}^{j,j}(t) - \hat{X}^{j,i}(t) \right|^2,
    \end{aligned}
\end{equation*}
where the last inequality is due to the exchangeability of $Z^i$ and the Young's inequality $a\cdot b \leq \delta |a|^2 + \frac{1}{4\delta}|b|^2$ for arbitrary $\delta > 0$.

Finally, we prove \eqref{eq:2nd2inlmm}. 
% Intuitively, the term  $\int_0^{\tfrac{N}{p}t}\mathbb{E}\left|\hat{X}^{j,j}(s) - \hat{X}^{j,i}(s) \right|^2 ds $ is small due to the construction of the random time change, the time-continuity of SDE (Lemma \ref{lmm:hatX_continuity}), and the law of large number in time. 
Define the stopping time (recall the definition of $\tau_n^i$ in \eqref{eq:defmniintro}, and set $\tau^i_{-1} = 0$ for all $i$)
\begin{equation*}
    n^i(s):=\sup \left\{n: \kappa \tau_n^i \leq s, n \in \mathbb{Z}, n\geq -1\right\}, \quad 1 \leq i \leq N, 
\end{equation*}
which is the total number of times the index $i$ is chosen into the batch before time $s$. Then using the random time change relation \eqref{eq:comparetwoIPS}, we know that for fixed $j$, for $s \in [t_k,t_{k+1})$, for any $i$,
\begin{equation}\label{eq:locationoft}
\begin{aligned}
    &\hat{X}^{j,i}(s) = X^j(n^i(t) \kappa + s-t_k),\quad \text{if} \quad i \in \mathcal{C}_k,\\
    &\hat{X}^{j,i}(s) = X^j(n^i(t) \kappa + \kappa),\quad \text{if} \quad i \notin \mathcal{C}_k.
\end{aligned}
\end{equation}
Also, from Lemma \ref{lmm:hatX_continuity}, we know that for any positive $t_1$, $t_2$,
\begin{equation}\label{eq:fromstabilitylmm}
    \mathbb{E}\left|X^\ell(t_2) - X^\ell(t_1) \right|^2 \leq 3\sigma^2 |t_2 - t_1|+C|t_2 - t_1|^2,
\end{equation}
where $C$ is a positive constant independent of $t_1$, $t_2$, $\kappa$, $\ell$. Now combining \eqref{eq:locationoft} and \eqref{eq:fromstabilitylmm}, using the assumed polynomial growth condition for $\nabla V$, the boundedness of $K$ and the moment bound in Lemma \ref{lmm:moment}, we have
\begin{equation}\label{eq:boundXbytau}
    \mathbb{E}\left|\hat{X}^{j,j}(s) - \hat{X}^{j,i}(s) \right|^2 \leq C\kappa\left(\mathbb{E}\left|n^i(s) - n^j(s) + 1 \right| + \kappa\mathbb{E}\left|n^i(s)- n^j(s) + 1 \right|^2\right).
\end{equation}
Note that 
\begin{multline}
    \kappa\mathbb{E}|n^i(s) - n^j(s)| \leq 2\kappa\sqrt{Var(n^1(s))} = 2\kappa \sqrt{(\lfloor \tfrac{s}{\kappa}\rfloor+1)\tfrac{p}{N}(1-\tfrac{p}{N})}\\
    \leq 2\kappa\sqrt{\tfrac{N}{p}\tfrac{t}{\kappa}\tfrac{p}{N}(1-\tfrac{p}{N})} \leq 2\sqrt{T\kappa}.
\end{multline}
Therefore, for small $\kappa$,
\begin{equation}
     \mathbb{E}\left|\hat{X}^{j,j}(s) - \hat{X}^{j,i}(s) \right|^2 \leq C\sqrt{T\kappa},\quad \forall s \in [0,T],
\end{equation}
where $C$ is a positive constant independent of $N$, $p$, $i$, $j$, $T$, $s$, $\kappa$.

\end{proof}

The following lemma controls the error brought by the random batch, and provides upper bounds for the third term on the right-hand side of \eqref{eq:splitintro}. The proof relies on the consistency of random batch, Lipschitz and boundedness assumption for $K(\cdot)$, and the continuity property of differential equations in Lemma \ref{lmm:hatX_continuity}.
\begin{lemma}\label{eq:deltakai}
Recall the definitions of $Z^i(t)$ and $\mathcal{X}(\tilde{X}(t))$:
\begin{equation*}
\begin{aligned}
     &Z^i(t) = \tilde{X}^i(t) - \hat{X}^{i,i}(t),\\
     &\mathcal{X}_k^i\left(\tilde{X}(t) \right):= \frac{1}{p-1}\sum_{j\in \mathcal{C}_k,j\neq i}K\left(\tilde{X}^j(t) - \tilde{X}^i(t) \right) - \frac{1}{N-1}\sum_{j\neq i}K\left(\tilde{X}^j(t) - \tilde{X}^i(t) \right).
\end{aligned}
\end{equation*}
Then for $i \in \mathcal{C}_k$, $s \in [t_k,t_{k+1})$,
\begin{enumerate}
\item $\mathbb{E}\left[Z^{i}\left(t_k\right) \cdot \mathcal{X}_k^i\left(\tilde{X}\left(t_k\right)\right)\right] = 0,$
\item $\mathbb{E}\left[\left(Z^i(s)-Z^i\left(t_k\right)\right) \cdot \mathcal{X}_k^i(\tilde{X}(s))\right]\leq C\kappa$
\item $\mathbb{E}\left[Z^i(t_k) \cdot\left(\mathcal{X}_k^i \left( \tilde{X}(s)\right)-\mathcal{X}_k^i(\tilde{X}(t_k))\right)\right] \leq C\mathbb{E}\left[|Z^i(s)|^2\right]^{\frac{1}{2}} \kappa + C\kappa^2,$
\end{enumerate}
where $C$ is independent of $i$, $k$, $N$, $p$, $s$. Consequently, by Young's inequality, for $\delta' > 0$ to be determined,
\begin{equation}\label{eq:3rdinlmm}
    \mathbb{E}\left[Z^i(s) \cdot \mathcal{X}_k^{i}(\tilde{X}(s))\right] \leq \delta'\mathbb{E}\left[|Z^i(s)|^2\right]  + \left(\frac{C^2}{4\delta'}+C\right)\kappa^2 + C\kappa,
\end{equation}
\end{lemma}

\begin{proof}
First, we prove the claim
\begin{equation}\label{eq:firstclaim}
\mathbb{E}\left[Z^{i}\left(t_k\right) \cdot \mathcal{X}_k^i\left(\tilde{X}\left(t_k\right)\right)\right] = 0
\end{equation}
Since at $t_k$, the random batch $\mathcal{C}_k$ is independent of $\tilde{X}^i(t_k)$ and $\hat{X}^{i,i}(t_k)$, and since
\begin{equation*}
\mathbb{E}\left[\mathcal{X}_k^i\left(\tilde{X}\left(t_k\right)\right)\right] = \mathbb{E}\left[\mathbb{E}\left[\mathcal{X}_k^i\left(\tilde{X}\left(t_k\right)\right) \mid \tilde{X}^i(t_k),1\leq i \leq N\right]\right]=0
\end{equation*}
by consistency of the random batch (recall the definition of $\mathcal{X}_k^i$ in the statement of Lemma \ref{eq:deltakai}, and the consistency property in \eqref{eq:unbiased}), we know that the first claim \eqref{eq:firstclaim} holds.

Next, we prove the second claim
\begin{equation}\label{eq:secondclaim}
    \mathbb{E}\left[\left(Z^i(s)-Z^i\left(t_k\right)\right) \cdot \mathcal{X}_k^i(\tilde{X}(s))\right]\leq C\kappa.
\end{equation}
Since $K(\cdot)$ is assumed to be bounded, we have
\begin{equation*}
\begin{aligned}
    \mathbb{E}\left[\left(Z^i(s)-Z^i\left(t_k\right)\right) \cdot \mathcal{X}_k^i(\tilde{X}(s))\right]
    &= \mathbb{E}\left[\left(Z^i(s)-Z^i\left(t_k\right)\right) \cdot \mathcal{X}_k^i(\tilde{X}(s))\right]\\
    &\leq 2 \|K\|_{\infty} \mathbb{E}\left|Z^i(s)-Z^i\left(t_k\right) \right|
\end{aligned}
\end{equation*}
By Lemma \ref{lmm:hatX_continuity}, 
\begin{equation*}
    \mathbb{E}\left|Z^i(s)-Z^i\left(t_k\right) \right|^2 \leq C\kappa^2,
\end{equation*}
where $C>0$ is independent of $s$, $k$, $i$, $N$, $p$. Hence by H\"older's inequality, the second claim \eqref{eq:secondclaim} holds.

Finally, we prove the third claim

\begin{equation}\label{eq:thirdclaim}
    \mathbb{E}\left[Z^i(t_k) \cdot\left(\mathcal{X}_k^i \left( \tilde{X}(s)\right)-\mathcal{X}_k^i(\tilde{X}(t_k))\right)\right] \leq C\mathbb{E}\left[|Z^i(s)|^2\right]^{\frac{1}{2}} \kappa + C\kappa^2.
\end{equation}
By Cauchy-Schwarz inequality,
\begin{multline}\label{eq:targetterm}
     \mathbb{E}\left[Z^i(t_k) \cdot\left(\mathcal{X}_k^i \left( \tilde{X}(s)\right)-\mathcal{X}_k^i(\tilde{X}(t_k))\right)\right]\\
     \leq \mathbb{E}\left[\left|Z^i(t_k)\right|^2 \right]^{\frac{1}{2}}  \mathbb{E}\left[\left|\mathbb{E}\left[\mathcal{X}_k^i \left( \tilde{X}(s)\right)-\mathcal{X}_k^i(\tilde{X}(t_k)) \mid \mathcal{F}_k\right] \right|^2 \right]^{\frac{1}{2}},
\end{multline}
where $\mathcal{F}_k := \sigma\left(\tilde{X}^i(0), \tilde{B}^i(s), \mathcal{C}_{k'}: s\leq t_k, k' \leq k, 1\leq i\leq N  \right)$.
By definition,
\begin{equation*}
    \mathbb{E}\left[\mathcal{X}_k^i \left( \tilde{X}(s)\right)-\mathcal{X}_k^i(\tilde{X}(t_k)) \mid \mathcal{F}_k\right] = \frac{1}{p-1}\sum_{j \in \mathcal{C}_k,j\neq i}\mathbb{E}\left[\Delta \tilde{K}^{ij}\mid \mathcal{F}_k\right] - \frac{1}{N-1}\sum_{j\neq i}\mathbb{E}\left[\Delta \tilde{K}^{ij}\mid \mathcal{F}_k\right],
\end{equation*}
where
\begin{equation*}
    \Delta \tilde{X}^i := \tilde{X}^i(s) - \tilde{X}^i(t_k),\quad 1\leq i \leq N,
\end{equation*}
and
\begin{equation}\label{eq:deltaK}
\begin{aligned}
    \Delta \tilde{K}^{ij} := K(\tilde{X}^j(s) - \tilde{X}^i(s)) - K(\tilde{X}^j(t_k) - \tilde{X}^i(t_k))
    \leq L\left|\Delta \tilde{X}^i - \Delta \tilde{X}^j \right|.
\end{aligned}
\end{equation}
Recall the time evolution of $\tilde{X}^i$ ($i \in \mathcal{C}_k$) in RBM-r \eqref{eq:RBM-rintro} for $s \in [t_k,t_{k+1})$:
\begin{equation*}
    \tilde{X}^i(s) = \tilde{X}^i(t_k)  -\int_{t_k}^s \nabla V(\tilde{X}^i(u)) du + \int_{t_k}^s  \frac{1}{p-1} \sum_{j\in \mathcal{C}_k,j\neq i} K(\tilde{X}^j(u) - \tilde{X}^i(u))du + \int_{t_k}^s \sigma d\tilde{B}^i_u 
\end{equation*}
Taking the conditional expectation, since $\nabla V(\cdot)$ grows at most polynomially, and since $K(\cdot)$ is assumed to be bounded, we have 
\begin{equation*}
\begin{aligned}
    \left|\mathbb{E}\left[\Delta \tilde{X}^{i}\mid \mathcal{F}_k\right]\right| &= \left|\int_{t_k}^s \mathbb{E}\left[\nabla V(\tilde{X}^i(u))\mid \mathcal{F}_k\right] du + \int_{t_k}^s \frac{1}{p-1} \sum_{j\in \mathcal{C}_k,j\neq i} K(\tilde{X}^j(u) - \tilde{X}^i(u))du\right|\\
    &\leq  C\left(\int_{t_k}^s \mathbb{E}\left[|\tilde{X}^i(u)|^q \mid \mathcal{F}_k\right]du + \kappa\right).
\end{aligned}
\end{equation*}
In the proof of Lemma \ref{lmm:moment} below (see \eqref{eq:importantinmoment}), we have 
\begin{equation*}
    \frac{d}{ds}\mathbb{E}\left[\left|\tilde{X}^{i}\right|^q \mid \mathcal{F}_k\right]\leq -C_1' \mathbb{E}\left[\left|\tilde{X}^{i}\right|^q \mid \mathcal{F}_k\right] + C_2'
\end{equation*}
for some $C_1'$, $C_2' > 0$ independent of $i$, $s$, $k$, $N$, $p$. Therefore for any $s \in [t_k,t_{k+1})$,
\begin{equation*}
    \mathbb{E}\left[|\tilde{X}^i(s)|^q \mid \mathcal{F}_k\right] \leq C_1 \left|\tilde{X}^i(t_k) \right|^q + C_2
\end{equation*}
for some $C_1$, $C_2 > 0$ independent of $i$, $s$, $k$, $N$, $p$. So we have the following estimate for $\mathbb{E}\left[\Delta \tilde{X}^{i}\mid \mathcal{F}_k\right]$:
\begin{equation}\label{eq:deltaX}
    \mathbb{E}\left[\Delta \tilde{X}^{i}\mid \mathcal{F}_k\right] \leq C\kappa\left(1 +  \left|\tilde{X}^i(t_k) \right|^q \right)
\end{equation}
for some $C > 0$ independent of $i$, $t$, $k$, $N$, $p$.

Combining \eqref{eq:targetterm}, \eqref{eq:deltaK}, \eqref{eq:deltaX}, and the uniform moment bound for $\tilde{X}$ in Lemma \ref{lmm:moment}, we have
\begin{equation*}
    \mathbb{E}\left[Z^i(t_k) \cdot\left(\mathcal{X}_k^i \left( \tilde{X}(s)\right)-\mathcal{X}_k^i(\tilde{X}(t_k))\right)\right]  \leq C\mathbb{E}\left[|Z^i(t_k)|^2\right]^{\frac{1}{2}} \kappa,
\end{equation*}
where $C$ is a positive constant independent of $s$, $k$, $i$, $\kappa$, $N$, $p$, $T$. Finally, by Lemma \ref{lmm:hatX_continuity},
\begin{equation*}
    \mathbb{E}\left|Z^i(s)-Z^i\left(t_k\right) \right|^2 \leq C\kappa^2.
\end{equation*}
Hence, we conclude that 
\begin{equation*}
    \mathbb{E}\left[Z^i(t_k) \cdot\left(\mathcal{X}_k^i \left( \tilde{X}(s)\right)-\mathcal{X}_k^i(\tilde{X}(t_k))\right)\right]  \leq C\mathbb{E}\left[|Z^i(s)|^2\right]^{\frac{1}{2}} \kappa + C\kappa^2,\quad \forall s \in [t_k,t_{k+1}),
\end{equation*}
where $C$ is a positive constant independent of $s$, $k$, $i$, $\kappa$, $N$, $p$, $T$.
\end{proof}

\section{Random Batch Method with replacement (RBM-r) and convergence analysis}\label{sec:mainresult}
In this section, we present the main result of this paper: convergence analysis for the Random Batch Method with replacement (RBM-r). In what follows, we first give the assumption required for convergence, and then present our main theorem (Theorem \ref{eq:thm:mainthm}), followed by some corollaries and discussions. At the end of this section, we give a proof of Theorem \ref{eq:thm:mainthm}, based on the lemmas in Section \ref{sec:lemmas}. To facilitate understanding, we begin with a brief summary of our proof before proceeding to the details.

Recall the following first-order interacting particle system (IPS): 
\begin{equation}\label{eq:IPSoriginal}
dX^i = -\nabla V(X^i) dt + \frac{1}{N-1}\sum_{j\neq i}K\left(X^j - X^i \right) dt + \sigma dB^i,\quad X^i|_{t=0} = X_0^i, \quad 1\leq i \leq N,
\end{equation}
where $X^i \in \mathbb{R}^d$ ($1\leq i \leq N$), $V: \mathbb{R}^d \rightarrow \mathbb{R}$ is the external potential, $K: \mathbb{R}^d \rightarrow \mathbb{R}^d$ is the interaction kernel, $B^i$ ($1 \leq i \leq N$) are independent Brownian motions in $\mathbb{R}^d$, and $\sigma \geq 0$ is the constant volatility. As introduced earlier in Section \ref{sec:intro}, denoting $t_k := k\kappa$ for $k = 0,1,2,\dots$ (where $\kappa$ the constant time step), choosing a random batch $\mathcal{C}_k$ with batch size $p$ at each $t = t_k$, the RBM-r algorithm is given by
\begin{equation}\label{eq:RBM-roriginal}
\begin{aligned}
    &\tilde{X}^i|_{t=0} = \tilde{X}_0^i,\\
    &d\tilde{X}^i = -\nabla V(\tilde{X}^i) dt + \frac{1}{p-1}\sum_{j\in \mathcal{C}_k,j\neq i}K\left(\tilde{X}^j - \tilde{X}^i \right) dt + \sigma d\tilde{B}^i,\quad i \in \mathcal{C}_k,\quad t \in [t_k,t_{k+1}),\\
    &d\tilde{X}^i = 0,\quad i \notin \mathcal{C}_k,\quad t \in [t_k,t_{k+1})
\end{aligned}
\end{equation}
In our result, we assume that IPS and RBM-r share the same initial state (i.e. $X_0^i = \tilde{X}_0^i$ for all $1 \leq i \leq N$). With other regular conditions assumed, we prove that for suitable running times (which will be further discussed below in Remark \ref{rmk:pseudotime}), as $\kappa$ tends to zero, the Wasserstein-2 distance between the first marginal distributions of these two systems converges to zero with an explicit convergence rate.

\subsection{Assumptions and statement of the main result}\label{sec:assresult}

The following assumptions for $V(\cdot)$ and $K(\cdot)$ are required for the proof of the convergence: 
\begin{assumption}\label{ass}
\begin{enumerate}
    \item $V(\cdot)$ is $\lambda$-strongly convex on $\mathbb{R}^d$ i.e. $x \mapsto V(x)-\frac{\lambda}{2}|x|^2$ is convex for some $\lambda > 0$, and $\nabla V$ has polynomial growth (i.e. $|\nabla V(x)|\leq C\left(1+|x|^q\right)$ for some $q>0$ ).
    \item $K(\cdot)$ is bounded and Lipschitz on $\mathbb{R}^d$ with Lipschitz constant $L$.
    \item $\lambda>2 L$.
\end{enumerate}
\end{assumption}
A simple example satisfying Assumption \ref{ass} is: $V(x) = 5 |x|^2$, and $K(x) = -\tfrac{x}{|x|^2} \tfrac{|x|^2}{(\delta^2 + |x|^2)} = - \frac{x}{\delta^2 + |x|^2} $ for some $\delta > 0$. Note that the interaction kernel of this kind is of practical use, since it can be viewed as the regularized version of some singular kernels in many models such as the Keller-Segel chemotaxis system (see for instance, Section 2.2 of \cite{wang2024deepparticle}).
We also remark here that although the convergence rate we obtained (i.e. $O(\kappa^{\frac{1}{4}})$) is worse than that of RBM-1 in \cite{jin2020random} (i.e. $O(\kappa^{\frac{1}{2}})$), our assumption here is weaker than that in \cite{jin2020random}. Actually, due to lower order of Taylor's expansion required in our analysis, we do not need (1) polynomial growth of $\nabla^2 V(\cdot)$, (2) Lipschitz of $\nabla K(\cdot)$ in our result, which are required in \cite{jin2020random}.

With Assumption \ref{ass}, we are able to obtain the following theorem, stating the convergence of RBM-r under Wasserstein-2 distance with an explicit rate.

\begin{theorem}\label{eq:thm:mainthm}
Fix $T>0$. Consider the interaction particle system (IPS) defined in \eqref{eq:IPSoriginal} and the corresponding Random Batch Method with replacement (RBM-r) defined in \eqref{eq:RBM-roriginal}. Let $\rho_t^{(1)}$, $\tilde{\rho}_{t}^{(1)}$ be the first marginal distribution at time $t$ for IPS and RBM-r, respectively. Suppose Assumption \ref{ass} holds. Suppose $X_0^i = \tilde{X}_0^i$ for all $1 \leq i \leq N$. Then there exists $\kappa_0>0$  such that for all $\kappa < \kappa_0$, it holds
\begin{equation}
    \sup_{0\leq t\leq T}W_2\left(\tilde{\rho}_{\bar{t}}^{(1)},\rho_t^{(1)}\right) \leq C\sqrt{1+T}\kappa^{\frac{1}{4}},
\end{equation}
where the positive constant $C$ is independent of $T$, $N$, $p$ and $\kappa$, and $\bar{t}:= \tfrac{N}{p}t$ denotes the pseudo-time for any physical time $t \in [0,T]$.
\end{theorem}

\begin{remark}[Discussion on different running times]\label{rmk:pseudotime}
Here we discuss a bit the reason why the two dynamics IPS and RBM-r should be evaluated under different running times. Recall that we evaluate IPS at $t$ while RBM-r at $\tfrac{N}{p}t$. This is because RBM-r is in fact moving forward along a pseudo-time axis.

One possible approach to understanding this is to compare the algorithm structures of RBM-r and RBM-1. In detail, as has been proved in \cite{jin2020random}, for the same running time $t$, the Random Batch Method without replacement (RBM-1) converges to IPS as $\kappa$ tends to zero. Hence, it is reasonable that its double-loop variant introduced near \eqref{eq:RBMr'}, should converge to IPS after the same running time $t$. Note that in both RBM-1 and its double loop variant \eqref{eq:RBMr'}, in each time interval, we update the system $\tfrac{N}{p}$ times. However, as a single-loop equivalent form, the dynamics in RBM-r is only updated once in each time interval. Therefore, compared with the time for RBM-1 and IPS, the time for RBM-r should be considered as a pseudo one with the scaling constant $\tfrac{N}{p}$.

Another approach to understanding this pseudo-time is from the $L^2$ weak law of large number. Precisely, for the original interacting particle system, suppose the $i$-th particle moves forward for time t. Meanwhile, for RBM-r, in each time interval of length $\kappa$, the probability for the $i$-th particle to move forward is $p/N$ (otherwise, it is frozen). Consequently, by the $L^2$ weak law of large number, if one supposes the $i$-th particle to move forward for time $t$, one should run the RBM-r for time $\tfrac{N}{p}t$.

\end{remark}

\begin{remark}[Discussion on the strong convexity condition in Assumption \ref{ass}]
In many related literature, the confining condition of the drift term (the strong convexity condition for the external potential $V(\cdot)$) is used to make the estimate uniform-in-time. However, under the current assumptions and proof framework, the error bound is still dependent of the length of the time interval $T$. Hence, it is straightforward to consider whether we still have short-time convergence without the strong convexity assumption. In fact, going through the proof, the only places where we use the strong convexity condition are: (1) the first term in \eqref{eq:combine1}; (2) the estimation of the moments in Lemma \ref{lmm:moment}. Hence, it is possible to replace this strong convexity condition with some less restrictive ones such as the Lipschitz condition for $\nabla V(\cdot)$. Then, we can still obtain the convergence following exactly the same proof, and the dependency of $T$ becomes exponential growth. This means, one positive effect of this strong convexity assumption is reducing the dependency of $T$ from exponential to only sublinear in the current result established in Theorem \ref{eq:thm:mainthm}.

\end{remark}

In \cite{jin2020random}, under a bit stronger assumptions, the convergence rate of RBM-1 was proved to be $O(\kappa^{\frac{1}{2}})$,which is better than our $O(\kappa^{\frac{1}{4}})$ rate for RBM-r. Technically, this is because when applying the $L^2$ weak law of large number, the H\"older's continuity property of differential equations (see Lemma \ref{lmm:hatX_continuity} below) is currently unavoidable. Therefore, the existence of Brownian motion would result in a worse convergence rate. However, if we instead consider the deterministic interaction particle system with the diffusion coefficient $\sigma = 0$, then the convergence rate can be improved.
We summarize this improvement in the following Theorem:
\begin{theorem}[Improved convergence rate for when $\sigma = 0$]\label{coro:sigma0}
Fix $T>0$. Consider the interaction particle system (IPS) defined in \eqref{eq:IPSoriginal} and the corresponding Random Batch Method with replacement (RBM-r) defined in \eqref{eq:RBM-roriginal}. Let $\rho_t^{(1)}$, $\tilde{\rho}_{t}^{(1)}$ be the first marginal distribution at time $t$ for IPS and RBM-r, respectively. Suppose Assumption \ref{ass} holds. Suppose $X_0^i = \tilde{X}_0^i$ for all $1 \leq i \leq N$. Suppose further that 
$$\sigma = 0.$$ 
Then there exists $\kappa_0>0$  such that for all $\kappa < \kappa_0$, it holds
\begin{equation}
    \sup_{0 \leq t \leq T }W_2\left(\tilde{\rho}_{\bar{t}}^{(1)},\rho_t^{(1)}\right) \leq C\sqrt{1+T}\kappa^{\frac{1}{2}},
\end{equation}
where the positive constant $C$ is independent of $T$, $N$, $p$ and $\kappa$, and $\bar{t}:= \tfrac{N}{p}t$ denotes the pseudo-time for any physical time $t \in [0,T]$.
\end{theorem}

The proof of Theorem \ref{coro:sigma0} is identical to that of Theorem \ref{eq:thm:mainthm}. The only difference is the estimation for the H\"older's continuity property in Lemma \ref{lmm:hatX_continuity}. Indeed, as we will see in Lemma \ref{lmm:hatX_continuity}, let $X_t$ solves a SDE driven by the Brownian motion with constant volatility and bounded / Lipschitz drift, then for any $t_1$, $t_2$, the H\"older's continuity property is given by
\begin{equation}
     \mathbb{E}|X(t_1) - X(t_2)|^2\lesssim |t_2-t_1|^2 + \sigma^2 |t_2 - t_1|.
\end{equation}
Obviously, when there is no diffusion (i.e. $\sigma \equiv 0$), the SDE is changed into ODE, and the H\"older's continuity property is then improved to
\begin{equation}
     \mathbb{E}|X(t_1) - X(t_2)|^2\lesssim |t_2-t_1|^2.
\end{equation}
This then leads to the improvement of the final convergence rate.

\subsection{Proof of the main result}\label{sec:proof}

In this section, we give the proof of our main result, Theorem \ref{eq:thm:mainthm}. To begin with, we give a high-level overview of our proof. 

As mentioned in Section \ref{sec:intro}, our estimate for $W_2(\tilde{\rho}_{\bar{t}}^{(1)},\rho_t^{(1)})$ (recalling $\bar{t} := \tfrac{N}{p}t$) is based on a coupling method. Let $\tilde{X}$, $X$ be two realizations of the RBM-r and IPS dynamics coupled in some specific way to be described below. To our knowledge, a direct calculation for $d(\tilde{X},X)$ for some suitable distance $d(\cdot,\cdot)$ is difficult. This then motivates us to introduce an intermediate dynamics IPS', which consists of $N$ copies of IPS after random time change. 

% In details, we will define the following dynamics triple $\left(\tilde{X}, \hat{X}, X\right)$ with the same initial state
% \begin{equation*}
%     \tilde{X}^\ell(0)= \hat{X}^{\ell,i}(0)= X^\ell(0), \quad \forall 1\leq l,i \leq N.
% \end{equation*} 
%  \textbf{RBM-r:} for $1\leq i \leq N$, $t \in[t_k,t_{k+1})$, $k = 0,1,2,\dots$,
% \begin{equation}\label{eq:RBM-r}
% \begin{aligned}
%     &d\tilde{X}^{i} = -\nabla V(\tilde{X}^{i}) dt + \frac{1}{p-1}\sum_{j\in \mathcal{C}_k,j\neq i}K\left(\tilde{X}^j - \tilde{X}^i \right) dt + \sigma d\tilde{B}^{i},\quad i \in \mathcal{C}_k,\\
%     &d\tilde{X}^i = 0,\quad i \notin \mathcal{C}_k,
% \end{aligned}
% \end{equation}
% \textbf{IPS':} for $1\leq i \leq N$, $1 \leq \ell \leq N$, $t \in[t_k,t_{k+1})$, $k = 0,1,2,\dots$,
% \begin{equation}\label{eq:IPS'}
% \begin{aligned}
%     &d\hat{X}^{\ell,i} = -\nabla V(\hat{X}^{\ell,i}) dt + \frac{1}{N-1}\sum_{j\neq l}K\left(\hat{X}^{j,i} - \hat{X}^{\ell,i} \right) dt + \sigma d\hat{B}^{\ell,i},\quad i \in \mathcal{C}_k,\\
%     &d\hat{X}^{\ell,i} = 0,\quad i \notin \mathcal{C}_k,
% \end{aligned}
% \end{equation}
% \textbf{IPS:} for $1\leq \ell \leq N$, $t \in[t_k,t_{k+1})$, $k = 0,1,2,\dots$,
% \begin{equation}\label{eq:IPS}
% dX^\ell = -\nabla V(X^\ell) dt + \frac{1}{N-1}\sum_{j\neq l}K\left(X^j - X^\ell \right) dt + \sigma dB^\ell,
% \end{equation}

In detail, recall the construction of the coupled dynamics RBM, IPS' and IPS defined in \eqref{eq:RBM-rintro} - \eqref{eq:IPSintro}. Also recall that in the construction, RBM-r \eqref{eq:RBM-rintro} and IPS' \eqref{eq:IPS'intro} share the same random batch $\mathcal{C}_k$ at each $k$-th interval. Moreover, the Brownian motions above are coupled in the way that for any $1 \leq \ell \leq N$, $1 \leq i \leq N$, $n = 0,1,2,\dots$, $\Delta t \in [0,\kappa)$,
\begin{equation}\label{eq:coupleBM}
\begin{aligned}
    &\hat{B}^{\ell,i}(\tau_n^{i} \kappa + \Delta t) = B^\ell(n\kappa+\Delta t),\\
    &\tilde{B}^i(\tau_n^{i} \kappa + \Delta t) = B^{i}(n\kappa + \Delta t),
\end{aligned}
\end{equation}
where for each fixed $i$, the stopping time $\tau_n^i$ is the $n$-th time the particle with index $i$ is chosen into the batch:
\begin{equation}\label{eq:defmni}
    \tau_n^i := \inf\left\{K: \sum_{k=0}^K \textbf{1}_{\{i \in \mathcal{C}_k \}} > n \right\}.
\end{equation}
Note that from the way we couple the Brownian motion in \eqref{eq:coupleBM}, it is clear that $\tilde{X}^i$ in RMB-r \eqref{eq:RBM-rintro} and $\hat{X}^{i,i}$ in IPS' \eqref{eq:IPS'intro} are synchronously coupled in that
\begin{equation*}
    \tilde{B}^i \equiv \hat{B}^{i,i}
\end{equation*}
for each $i$. Moreover, from the definition of the RBM-r and the way we couple the Brownian motions, we see that IPS' \eqref{eq:IPS'intro} defines $N$ pairwise independent interacting particle systems: for each fixed $i$, all particles $\hat{X}^{\ell,i}$ ($1\leq \ell \leq N$) are moving forward (with the same interaction force) or staying still simultaneously according to whether the event $\{i \in \mathcal{C}_k\}$ happens at $k$-th interval. Consequently, we find that each $N$-body system $(\hat{X}^{\ell,i})_{\ell=1}^N$ is indeed a random time change from $(X^\ell)_{\ell=1}^N$ in IPS \eqref{eq:IPSintro}, induced by index $i$ in the sense that
\begin{equation}\label{eq:comparetwoIPS}
    \hat{X}^{\ell,i}(\tau_n^i \kappa + \Delta t) = X^\ell(n\kappa + \Delta t), \quad \forall \,\Delta t\in [0,\kappa), \quad n = 0,1,2,\dots,\quad 1\leq \ell \leq N.
\end{equation}
This random time change relation is essential in the later analysis. Another key observation for such construction is that it preserves the exchangeability of $Z^i := \hat{X}^{i,i} - \tilde{X}^i$, which means $Z^1, \dots, Z^N$ has the same distribution, Therefore, in the practical calculation, we only make use of the diagonal elements of IPS' (namely, $\hat{X}^{i,i}$ for $1 \leq i \leq N$) for comparisons.

After constructing the coupling above, for $0\leq t \leq T$, we estimate the distance between IPS and RBM-r by the following two main steps. In what follows, we will consider $t \in [0,T]$ and assume that $\kappa$ divides $t$ for simplicity.
\begin{itemize}
    \item Compare IPS and IPS': estimate $\mathbb{E}\left|\hat{X}^{i,i}(\tfrac{N}{p}t)  - X^i(t)\right|^2$.
    \item Compare IPS' and RBM-r: estimate $\mathbb{E}\left|\tilde{X}^i(\tfrac{N}{p}t) - \hat{X}^{i,i}(\tfrac{N}{p}t) \right|^2$.
\end{itemize}

To compare IPS and IPS', we will make use of random time change relation \eqref{eq:comparetwoIPS}, the continuity property proved in Lemma \ref{lmm:hatX_continuity} and the $L^2$ weak law of large numbers. After careful calculation, we will arrive at
\begin{equation*}
    \mathbb{E}\left|\hat{X}^{i,i}(\tfrac{N}{p}t)  - X^i(t)\right|^2 = \mathbb{E}\left|\hat{X}^{i,i}(\tfrac{N}{p}t)  - \hat{X}^{i,i}(\kappa \tau^i_{n_t})\right|^2 \lesssim \sqrt{\kappa},
\end{equation*}
where $n_t := t / \kappa$ (recall that we assume $\kappa$ divides $t$).

Comparing IPS' and RBM-r is more complicated. Define $Z^i = \tilde{X}^i - \hat{X}^{i,i}$ for $1\leq i \leq N$. Our strategy is to decompose $\frac{d}{ds}\mathbb{E}|Z^i(s)|^2$ into the following three parts:
\begin{equation*}
\begin{aligned}
    \frac{d}{ds}\mathbb{E}\left|Z^i(s) \right|^2 &= -\mathbb{E} \left[Z^i(s) \cdot \left(\nabla V\left(\tilde{X}^i(s) \right) - \nabla V\left(\hat{X}^{i,i}(s) \right)\right)\right]\\
    &\quad+ \frac{1}{N-1}\sum_{j \neq i}\mathbb{E}\left[Z^i(s) \cdot \left( K \left(\tilde{X}^j(s) - \tilde{X}^i(s) \right) - K\left( \hat{X}^{j,i}(s) - \hat{X}^{i,i}(s)\right)\right)\right]\\
    &\quad+ \mathbb{E}\left[Z^i(s) \cdot \mathcal{X}_k^i(\tilde{X}(s))\right],
\end{aligned}
\end{equation*}
where
\begin{equation*}
    \mathcal{X}_k^i\left(\tilde{X}(s) \right):= \frac{1}{p-1}\sum_{j\in \mathcal{C}_k,j\neq i}K\left(\tilde{X}^j(s) - \tilde{X}^i(s) \right) - \frac{1}{N-1}\sum_{j\neq i}K\left(\tilde{X}^j(s) - \tilde{X}^i(s) \right).
\end{equation*}
Employing the strong convexity of $V(\cdot)$, it is straightforward to estimate the first term on the right-hand side of \eqref{eq:splitintro} above. We handle the second term above by making use of the Lipschitz assumption for the interaction kernel $K(\cdot)$, the exchangeability of $Z$, and the $L^2$ weak law of large numbers. See more details below and in Lemma \ref{lmm:exchange}. The third term above is the error induced by the random batch. We would estimate it via the first-order Taylor's expansion and make use of the consistency of random batch $\mathcal{C}_k$ as well as the Lipschitz assumption for $K(\cdot)$. See more details below and in Lemma \ref{eq:deltakai}.

\begin{proof}[Proof of Theorem \ref{eq:thm:mainthm} assuming Lemma \ref{lmm:moment} - \ref{eq:deltakai}]
Fix $T > 0$. Let $\mathcal{C}_k$ be the random batch at $k$-th time interval $[t_k,t_{k+1})$ ($t_k := k\kappa$ with $\kappa$ being the constant time step).

\textbf{STEP 1: construction of coupling.}

As discussed before, the first step is to construct suitable coupling, and our main approach is introducing an intermediate dynamics IPS'. Recall the dynamics triple $\left(\tilde{X}, \hat{X}, X\right)$ with the same initial defined in \eqref{eq:RBM-rintro} - \eqref{eq:IPSintro}. Also recall that RBM-r \eqref{eq:RBM-rintro} and IPS' \eqref{eq:IPS'intro} share the same random batch $\mathcal{C}_k$ at each $k$-th interval. Furthermore, recalling the definition of $\tau_n^i$ in \eqref{eq:defmni}, the Brownian motions above satisfy that for any $1 \leq \ell \leq N$, $1 \leq i \leq N$, $n = 0,1,2,\dots$, $\Delta t \in [0,\kappa)$,
\begin{equation}
\begin{aligned}
    &\hat{B}^{\ell,i}(\tau_n^{i} \kappa + \Delta t) = B^\ell(n\kappa+\Delta t),\\
    &\tilde{B}^i(\tau_n^{i} \kappa + \Delta t) = B^{i}(n\kappa + \Delta t).
\end{aligned}
\end{equation}
In particular,
\begin{equation}
    \tilde{B}^i \equiv \hat{B}^{i,i}
\end{equation}
for each $i$. Consequently, the following random time change relation holds
\begin{equation}
    \hat{X}^{\ell,i}(\tau_n^i \kappa + \Delta t) = X^\ell(n\kappa + \Delta t), \quad \forall \,\Delta t\in [0,\kappa), \quad n = 0,1,2,\dots,\quad 1\leq \ell \leq N.
\end{equation}
Now with the constructed dynamics triple $\left(\tilde{X}, \hat{X}, X\right)$, for fixed $t \in [0,T]$, we are ready to compare RBM-r \eqref{eq:RBM-rintro} at time $\tfrac{N}{p}t$ and IPS \eqref{eq:IPSintro} at time $t$. In the following we would consider $t \in [0,T]$ and assume $\kappa$ divides $t$ for simplicity. We estimate by two main steps:

\textbf{STEP 2: compare IPS \eqref{eq:IPSintro} and IPS' \eqref{eq:IPS'intro} - estimate $\mathbb{E}\left|\hat{X}^{i,i}(\tfrac{N}{p}t)  - X^i(t)\right|^2$.}

Using the random time change relation \eqref{eq:comparetwoIPS}, and denoting $n_t := t/\kappa$ (recall that we assume $\kappa$ divides $t$), we know that
\begin{equation*}
    \mathbb{E}\left|\hat{X}^{i,i}(\tfrac{N}{p}t)  - X^i(t)\right|^2 = \mathbb{E}\left|\hat{X}^{i,i}(\tfrac{N}{p}t)  - \hat{X}^{i,i}(\kappa \tau^i_{n_t})\right|^2.
\end{equation*}
Denote 
\begin{equation}
    \hat{b}^i_s := -\nabla V(X^{i,i}(s)) + \frac{1}{N-1}\sum_{j \neq i}K(\hat{X}^{j,i}(s) - \hat{X}^{i,i}(s)),
\end{equation}
By definition, since in the time interval $[t_k,t_{k+1})$, $d\hat{X}^{i,i} = 0$ if $i \notin \mathcal{C}_k$, we have
\begin{equation}
    \mathbb{E}\Big|\hat{X}^{i,i}(\tfrac{N}{p}t)  - \hat{X}^{i,i}(\kappa \tau^i_{n_t})\Big|^2 = \mathbb{E}\Big|\int_{\kappa \tau^i_{n_t}}^{\tfrac{N}{p}t} \textbf{1}_{\{i\in \mathcal{C}_{\lfloor s/\kappa \rfloor}\}}(s)(\hat{b}^i(s) ds + \sigma dB_s) \Big|^2.
\end{equation}
Next, we estimate the above by taking into account the time intervals where the dynamics $\hat{X}^{i,i}$ is frozen. This means that the effective time interval in the integral above is in fact less than the total length $\tfrac{N}{p}t - \kappa \tau^i_{n_t}$. In detail, defined the event
\begin{equation}
    A(i,m,M) := \{\text{index i is chosen into the batch m times among M picks} \}\quad 
\end{equation}
with $1 \leq i \leq N$, $0 \leq m \leq M$.
Clearly, the probability of the event $A(i,m,M)$ can be written out explicitly:
\begin{equation}
    P\left(A(i,m,M) \right) = \binom{M}{m} (\tfrac{p}{N})^m(1 - \tfrac{P}{N})^{M-m}.
\end{equation}
Now, using the polynomial bound of $\nabla V(\cdot)$ as well as the boundedness of $K(\cdot)$ in Assumption \ref{ass}, we have
\begin{equation}\label{eq:eqinSTEP2}
\begin{aligned}
    &\quad\mathbb{E}\Big|\hat{X}^{i,i}(\tfrac{N}{p}t)  - \hat{X}^{i,i}(\kappa \tau^i_{n_t})\Big|^2\\
    &=\sum_{n = n_t}^\infty P(\tau^i_{n_t} = n) \mathbb{E}\Big|\int_{n\kappa}^{\tfrac{N}{p}t} \textbf{1}_{i\in \mathcal{C}_{\lfloor s/\kappa \rfloor}}(s)(\hat{b}^i(s) ds + \sigma dB_s)\Big|^2\\
    &\leq \sum_{n = n_t}^\infty P(\tau^i_{n_t} = n)  \Big(\sigma^2\big(\sum_{m=0}^{\tfrac{N}{p}n_t - n} m\kappa P(A(i,m,|\tfrac{N}{p}n_t - n|)) \big)\\
    &\quad+ C\big(\sum_{m=0}^{\frac{N}{p}n_t - n} m\kappa P(A(i,m,|\frac{N}{p}n_t - n|)) \big)^2\Big),
\end{aligned}
\end{equation}
where $C$ is a positive constant independent of $N$, $p$, $\kappa$, $t$, $T$, $i$. Moreover,
\begin{equation}
\begin{aligned}
&\quad\sum_{m=0}^{\tfrac{N}{p}n_t - n} m\kappa P(A(i,m,|\tfrac{N}{p}n_t - n|))\\
&= \sum_{m=0}^{\tfrac{N}{p}n_t - n} m\kappa \binom{|\tfrac{N}{p}n_t - n|}{m} (\tfrac{p}{N})^m(1 - \tfrac{p}{N})^{|\tfrac{N}{p}n_t - n|-m}\\
&=\kappa |\tfrac{N}{p}n_t - n| \tfrac{p}{N}.
\end{aligned}
\end{equation}
Therefore,
\begin{equation}
\mathbb{E}\left|\hat{X}^{i,i}(\tfrac{N}{p}t)  - \hat{X}^{i,i}(\kappa \tau^i_{n_t})\right|^2 \leq \sigma^2 \kappa \tfrac{p}{N}\mathbb{E}|\tau^i_{t_n} - \tfrac{N}{p}n_t| + C \kappa^2 \tfrac{p^2}{N^2}\mathbb{E}|\tau^i_{t_n} - \tfrac{N}{p}n_t|^2
\end{equation}
By Lemma \ref{eq:L2LLN}, the random variable $\tau^i_{n_t}$ has mean $\frac{N}{p}n_t$ and variance $((\frac{N}{p})^2 - \frac{N}{p})n_t$. Consequently, we have
\begin{equation}
\mathbb{E}\left|\hat{X}^{i,i}(\tfrac{N}{p}t)  - \hat{X}^{i,i}(\kappa \tau^i_{n_t})\right|^2 \leq \sigma^2  \kappa \tfrac{p}{N} \sqrt{Var(\tau^i_{n_t})} + C \kappa^2 \tfrac{p^2}{N^2} Var(\tau^i_{n_t}) \leq  \sigma^2\sqrt{T\kappa} + CT\kappa.
\end{equation}
In particular, when $\sigma \equiv 0$,
\begin{equation}
\mathbb{E}\left|\hat{X}^{i,i}(\tfrac{N}{p}t)  - \hat{X}^{i,i}(\kappa \tau^i_{n_t})\right|^2 \leq CT\kappa,
\end{equation}
where $C$ is a positive constant independent of $N$, $p$, $\kappa$, $t$, $T$, $i$.

\textbf{STEP 3: compare IPS' \eqref{eq:IPS'intro} and RBM-r \eqref{eq:RBM-rintro} - estimate $\mathbb{E}\left|\tilde{X}^i(\tfrac{N}{p}t) - \hat{X}^{i,i}(\tfrac{N}{p}t) \right|^2$.}

Define the process 
\begin{equation}\label{eq:Zidef}
    Z^i(s) := \tilde{X}^i(s) - \hat{X}^{i,i}(s)
\end{equation}
for $1 \leq i \leq N$ and $t \geq 0$. By Lemma \ref{lmm:exchange}, $(Z^i)_{i=1}^N$ are exchangeable. Recalling that we couple the Brownian motions $\tilde{B}^i$ and $\hat{B}^{i,i}$ synchronously, for $s \in [t_k,t_{k+1})$, we have
\begin{multline*}
    dZ^i(s) = - \left(\nabla V(\tilde{X}^i(s)) - \nabla V(\hat{X}^{i,i}(s)) \right) ds\\
    + \frac{1}{p-1}\sum_{j\in \mathcal{C}_k,j\neq i}K\left(\tilde{X}^j(s) - \tilde{X}^i(s) \right) ds - \frac{1}{N-1}\sum_{j\neq i}K\left(\hat{X}^{j,i}(s) - \hat{X}^{i,i}(s) \right) ds.
\end{multline*}
In order to estimate $Z^i$, our strategy is to decompose the above by:
\begin{equation*}
\begin{aligned}
    dZ^i(s) &= - \left(\nabla V(\tilde{X}^i(s)) - \nabla V(\hat{X}^{i,i}(s)) \right) ds\\
    &\quad+\left[\frac{1}{p-1}\sum_{j\in \mathcal{C}_k,j\neq i}K\left(\tilde{X}^j(s) - \tilde{X}^i(s) \right) - \frac{1}{N-1}\sum_{j\neq i}K\left(\tilde{X}^j(s) - \tilde{X}^i(s) \right)\right]ds\\
    &\quad + \left[\frac{1}{N-1}\sum_{j \neq i}  K \left(\tilde{X}^j(s) - \tilde{X}^i(s) \right) - \frac{1}{N-1}\sum_{j \neq i}K\left( \hat{X}^{j,i}(s) - \hat{X}^{i,i}(s)\right)\right]ds.
\end{aligned}
\end{equation*}
Intuitively, the first term in the dynamics above is easy to handle since we have assumed strong convexity for $V$ in Assumption \ref{ass}. For the second term, we will estimate it using the consistency of the random batch and the Lipschitz condition for the interaction kernel. For the third term above, we will make use of the exchangeability of $Z^i$, the Lipschitz condition for the interaction kernel, and the $L^2$ weak law of large numbers. Details are presented below.

Recalling the notation
\begin{equation*}
    \mathcal{X}_k^i\left(\tilde{X}(s) \right):= \frac{1}{p-1}\sum_{j\in \mathcal{C}_k,j\neq i}K\left(\tilde{X}^j(s) - \tilde{X}^i(s) \right) - \frac{1}{N-1}\sum_{j\neq i}K\left(\tilde{X}^j(s) - \tilde{X}^i(s) \right),
\end{equation*}
and taking expectation, for $s \in [t_k,t_{k+1})$, we have
\begin{equation}\label{eq:combine1}
\begin{aligned}
    \frac{d}{ds}\mathbb{E}\left|Z^i(s) \right|^2 &= -\mathbb{E} \left[Z^i(s) \cdot \left(\nabla V\left(\tilde{X}^i(s) \right) - \nabla V\left(\hat{X}^{i,i}(s) \right)\right)\right]\\
    &\quad+ \frac{1}{N-1}\sum_{j \neq i}\mathbb{E}\left[Z^i(s) \cdot \left( K \left(\tilde{X}^j(s) - \tilde{X}^i(s) \right) - K\left( \hat{X}^{j,i}(s) - \hat{X}^{i,i}(s)\right)\right)\right]\\
    &\quad+ \mathbb{E}\left[Z^i(s) \cdot \mathcal{X}_k^i(\tilde{X}(s))\right]
\end{aligned}
\end{equation}

Next, we estimate three terms in \eqref{eq:combine1} above separately. For the first term, by strongly convexity of $V(\cdot)$ in Assumption \ref{ass}, 
\begin{equation}\label{eq:combine2}
    -\mathbb{E} \left[Z^i(s) \cdot \left(\nabla V\left(\tilde{X}^i(s) \right) - \nabla V\left(\hat{X}^{i,i}(s) \right)\right)\right] \leq - \lambda \mathbb{E}\left|Z^i(s) \right|^2.
\end{equation}

For the second term in \eqref{eq:combine1}, mainly making use of the Lipschitz condition of $K(\cdot)$ in Assumption \ref{ass}, the exchangeability of $Z^i$, and the $L^2$ weak law of large number, we prove in Lemma \ref{lmm:exchange} below that
\begin{enumerate}
    \item For arbitrary $\delta > 0$ to be determined,
    \begin{multline}\label{eq:2nd1}
    \frac{1}{N-1}\sum_{j \neq i}\mathbb{E}\left[Z^i(s) \cdot \left( K \left(\tilde{X}^j(s) - \tilde{X}^i(s) \right) - K\left( \hat{X}^{j,i}(s) - \hat{X}^{i,i}(s)\right)\right)\right]\\
    \leq 2L\mathbb{E}\left|Z^i(s) \right|^2 + \delta \mathbb{E}\left|Z^i(s) \right|^2 + \frac{1}{4\delta}\frac{1}{N-1}\sum_{j \neq i} \mathbb{E}\left|(\hat{X}^{j,j}(s) - \hat{X}^{j,i}(s) \right|^2,
    \end{multline}
    \item For small $\kappa$, we have
    \begin{equation}\label{eq:2nd2}
        \mathbb{E}\left|\hat{X}^{j,j}(s) - \hat{X}^{j,i}(s) \right|^2 \leq C\sqrt{T\kappa}, \forall s \in [0,T],
    \end{equation}
where $C$ is independent of $i$, $j$, $k$, $N$, $p$, $\kappa$, $T$, $t$.
\end{enumerate}

For the third term in \eqref{eq:combine1}, by consistency of the random batch and the Lipschitz assumption of $K(\cdot)$, we further decompose it into
\begin{multline*}
 \mathbb{E}\left[Z^i(s) \cdot \mathcal{X}_k^{i}(\tilde{X}(s))\right]
 =\mathbb{E}\left[Z^{i}\left(t_k\right) \cdot \mathcal{X}_k^i\left(\tilde{X}\left(t_k\right)\right)\right]\\
 + \mathbb{E}\left[\left(Z^i(s)-Z^i\left(t_k\right)\right) \cdot \mathcal{X}_k^i(\tilde{X}(s))\right]\
 +\mathbb{E}\left[Z^i(t_k) \cdot\left(\mathcal{X}_k^i \left( \tilde{X}(s)\right)-\mathcal{X}_k^i(\tilde{X}(t_k))\right)\right],
\end{multline*}
and prove in Lemma \ref{eq:deltakai} below that
\begin{enumerate}
\item $\mathbb{E}\left[Z^{i}\left(t_k\right) \cdot \mathcal{X}_k^i\left(\tilde{X}\left(t_k\right)\right)\right] = 0,$
\item $\mathbb{E}\left[\left(Z^i(s)-Z^i\left(t_k\right)\right) \cdot \mathcal{X}_k^i(\tilde{X}(s))\right]\leq C\kappa$
\item $\mathbb{E}\left[Z^i(t_k) \cdot\left(\mathcal{X}_k^i \left( \tilde{X}(s)\right)-\mathcal{X}_k^i(\tilde{X}(t_k))\right)\right] \leq C\mathbb{E}\left[|Z^i(s)|^2\right]^{\frac{1}{2}} \kappa + C\kappa^2,$
\end{enumerate}
where $C$ is independent of $i$, $s$, $k$, $N$, $p$. Consequently, by Young's inequality, for arbitrary $\delta' > 0$ to be determined,
\begin{equation*}
    C\mathbb{E}\left[|Z^i(s)|^2\right]^{\frac{1}{2}} \kappa \leq \delta' \mathbb{E}\left[|Z^i(s)|^2\right] + \frac{C^2}{4\delta'}\kappa^2.
\end{equation*}
Therefore,
\begin{equation}\label{eq:3rd}
    \mathbb{E}\left[Z^i(s) \cdot \mathcal{X}_k^{i}(\tilde{X}(s))\right] \leq \delta'\mathbb{E}\left[|Z^i(s)|^2\right]  + \left(\frac{C^2}{4\delta'}+C\right)\kappa^2 + C\kappa,
\end{equation}

Finally, combining \eqref{eq:combine2}, \eqref{eq:2nd1}, \eqref{eq:2nd2} and \eqref{eq:3rd}, and choosing $\delta$, $\delta'$ small enough such that $\delta+\delta'<\lambda - 2L$, we conclude from Gr\"onwall's inequality that for small $\kappa$,
\begin{equation*}
    \mathbb{E}\left|\tilde{X}^i(\tfrac{N}{p}t) - \hat{X}^{i,i}(\tfrac{N}{p}t) \right|^2 
    \leq C(1+\sqrt{T})\sqrt{\kappa}.
\end{equation*}

Eventually, by STEP 2 and STEP 3, we obtain
\begin{equation*}
    \mathbb{E}\left|\tilde{X}^i(\tfrac{N}{p}t) - X^i(t) \right|^2 \leq C(1+\sqrt{T})\sqrt{\kappa}.
\end{equation*}
Consequently,
\begin{equation*}
    \sup_{0\leq t  \leq T}W_2(\tilde{\rho}_{\tfrac{N}{p}t}^{(1)},\rho_t^{(1)}) \leq C(1+T^{\frac{1}{4}})\kappa^{\frac{1}{4}},
\end{equation*}
where all constants $C$ are independent of $T$, $N$, $p$ and $\kappa$.

\end{proof}

% \begin{remark}[Pairwise independency of IPS']
% We remark here that in our proof, the $N$ systems constructed in IPS' \eqref{eq:IPS'} are not independent. This is because for fixed $k$, the events $\left(\{i \in \mathcal{C}_k \}\right)_{i=1}^N$ are only pairwise independent due to the existence of the size restriction of the random batch $|\mathcal{C}_k|=p$ ($p \geq 2$). However, it is easy to see that throughout our analysis, only pairwise independence of the events is required. In particular,
% \end{remark}

\section{Numerical example}
In this section, we perform a simple numerical experiment to showcase the efficiency and the convergence property of the RBM-r algorithm, even in the case of singular interaction kernels. We consider a $N$-particle system with Brownian motion in the following form:
\begin{equation}\label{eq:psnumerical}
    dX_i = -\beta X_i dt + \frac{1}{N}\sum_{j \neq i} \frac{X_i - X_j}{|X_i - X_j|^2} dt + \frac{1}{\sqrt{N}} dB^i,\quad 1 \leq i \leq N.
\end{equation}
Here, $\beta \in \mathbb{R}_{+}$, $X_i \in \mathbb{R}^d$, and $B^i$ are $N$ independent Brownian motions in $\mathbb{R}^d$. In our experiment, we choose $d = 2$, $\beta = 1$, $N = 10^3$. The initial distribution is chosen to be standard Gaussian. We simulate both the original particle system \eqref{eq:psnumerical} and the associated RBM-r dynamics using the forward Euler discretization with a  time step $\kappa_{Euler} = 10^{-2}$. The time step for choosing the random batch is $\kappa = 10^{-1}$. In Figure \ref{fig:particle} -- Figure \ref{fig:distribution2}, we report the simulation results for both RBM-r dynamics and the original dynamics, including the particle position, the density, and the marginal densities. As we can see from the numerical results, even in the case of singular kernels, RBM-r can still perform a relatively good approximation of the original dynamics, with much lower computational cost.

\begin{figure}[htbp]
    \centering
    
    \includegraphics[width=0.6\textwidth]{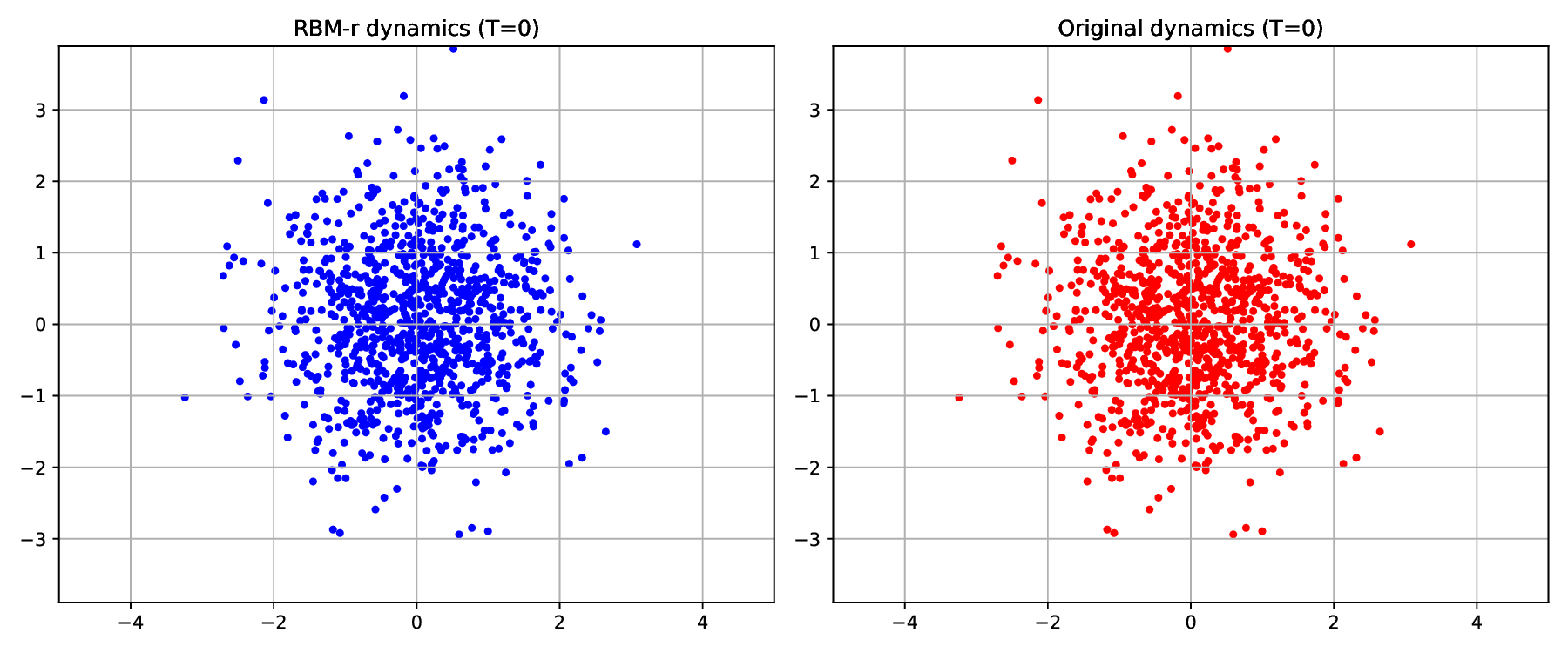}    

    \includegraphics[width=0.6\textwidth]{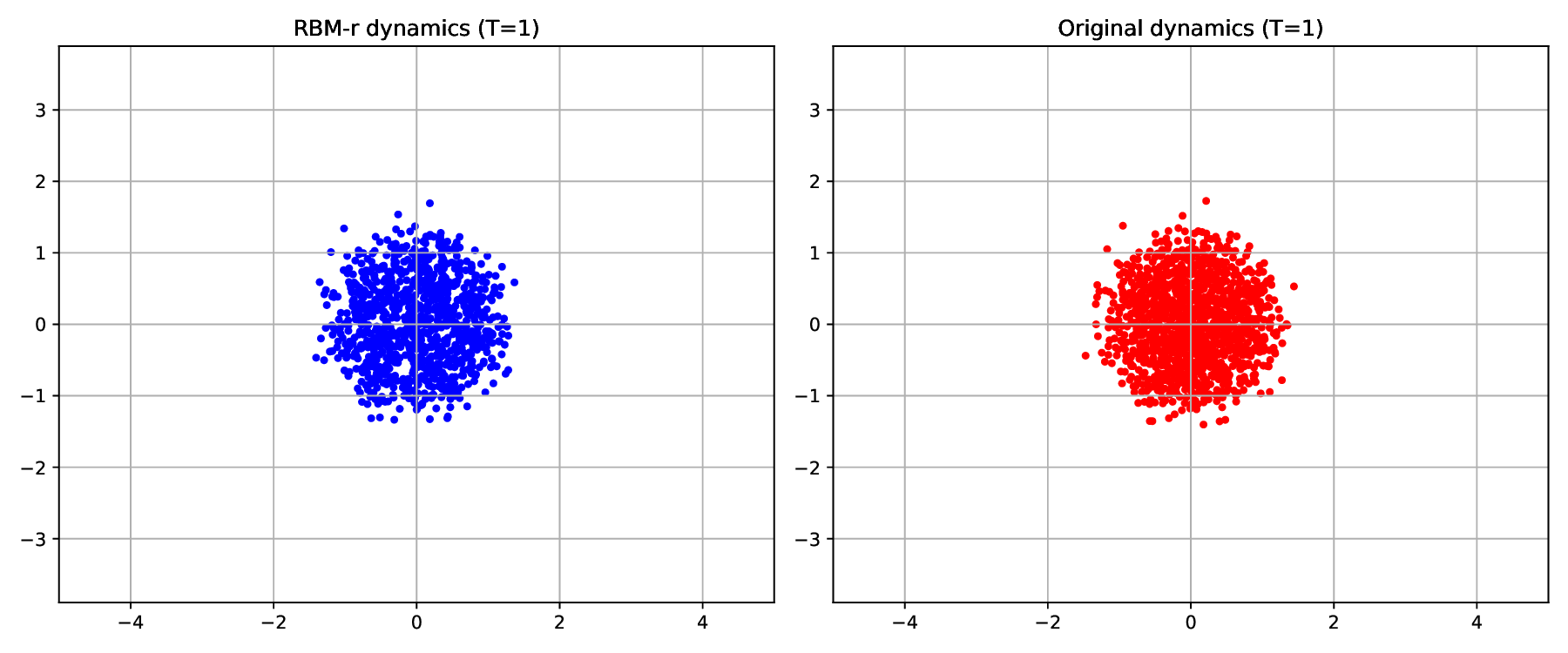}

    \includegraphics[width=0.6\textwidth]{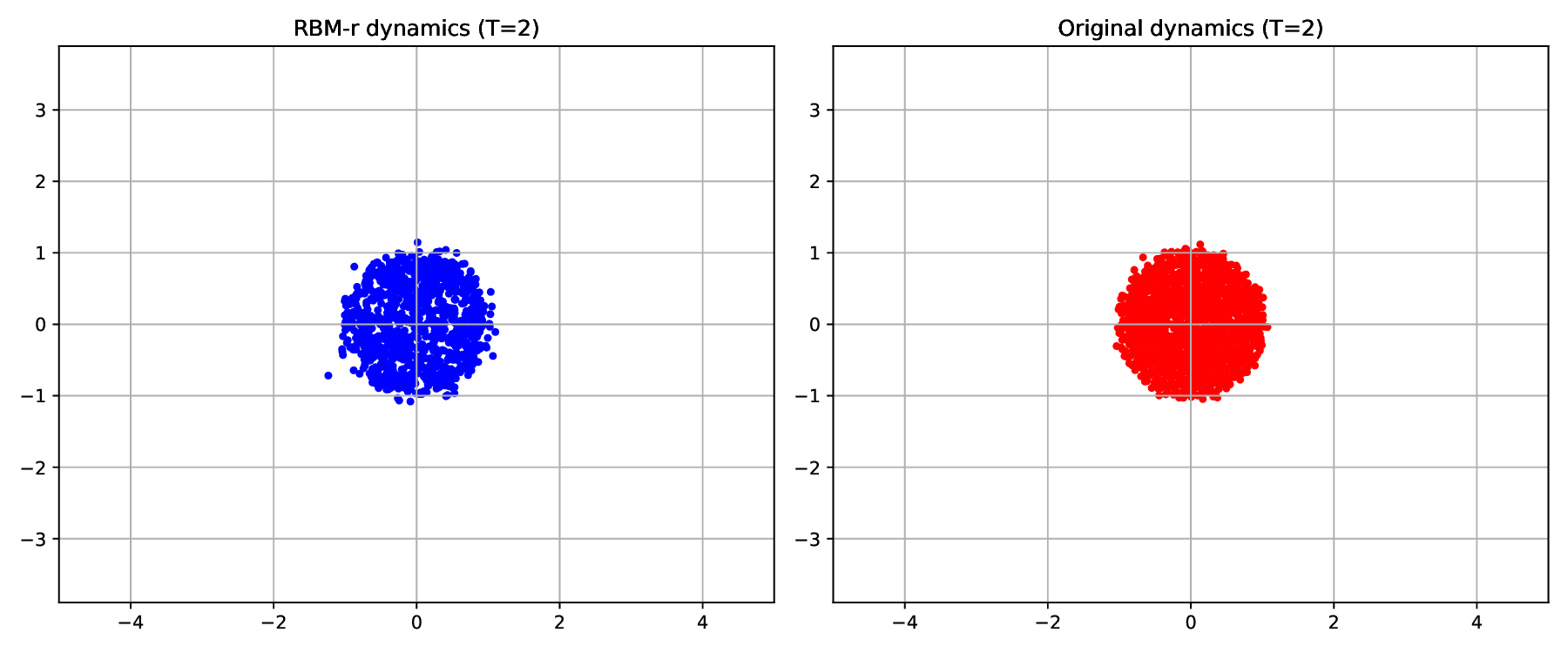}

\caption{Evolution of particle positions of RBM-r dynamics and original dynamics}
\label{fig:particle}
\end{figure}

\begin{figure}[htbp]
    \centering
    
    \includegraphics[width=0.6\textwidth]{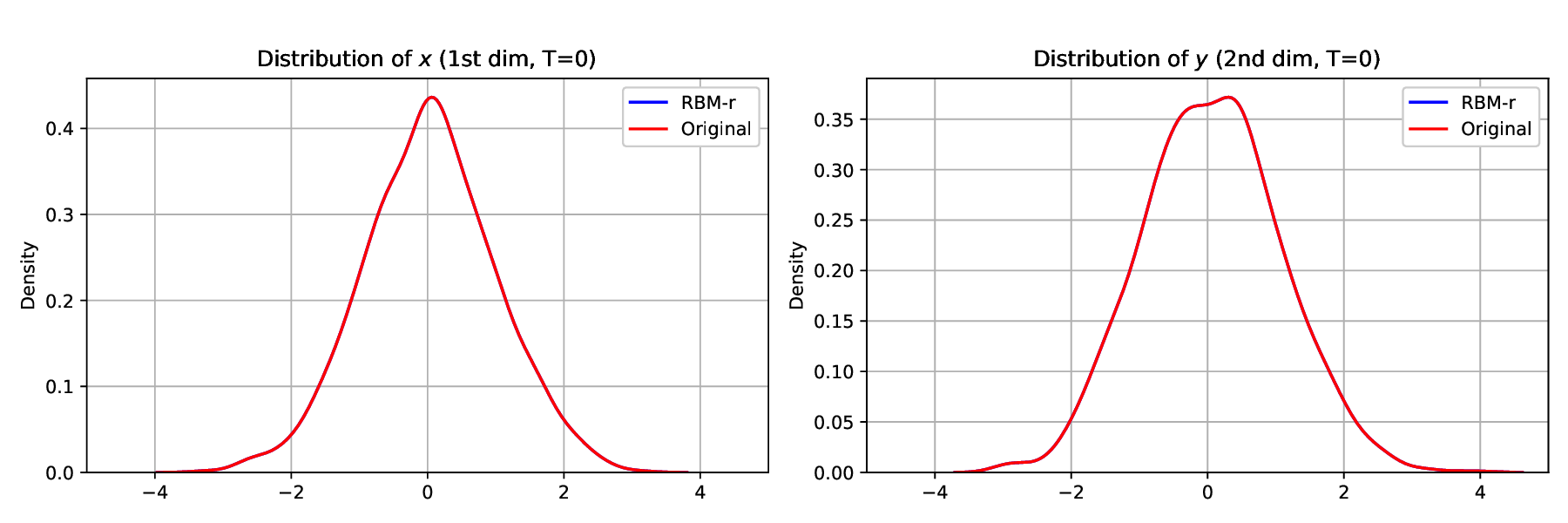}    

    \includegraphics[width=0.6\textwidth]{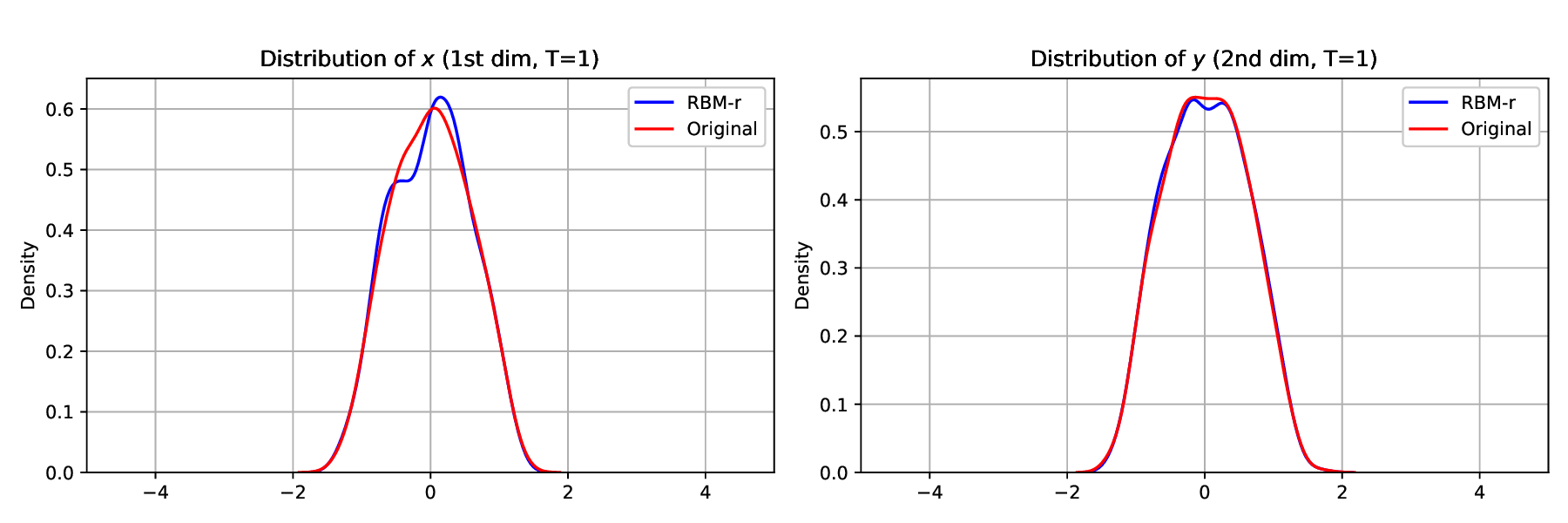}

    \includegraphics[width=0.6\textwidth]{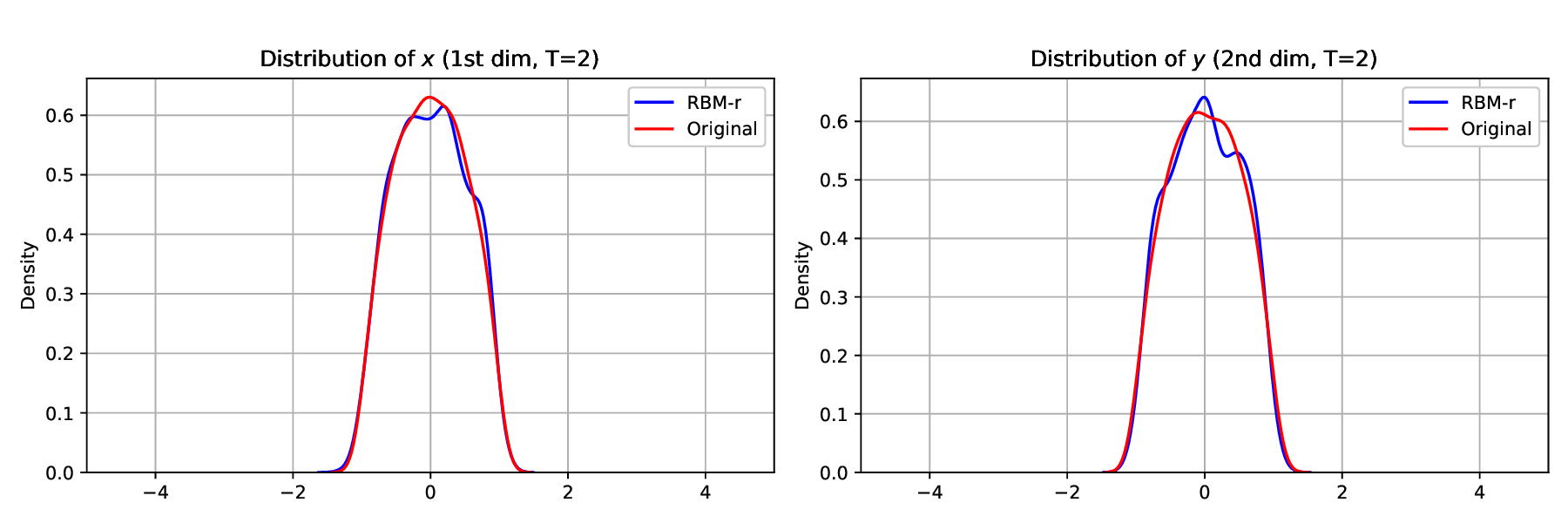}

\caption{Evolution of particle marginal distributions of RBM-r dynamics and original dynamics}
\label{fig:distribution}
\end{figure}

\begin{figure}[htbp]
    \centering
    
    \includegraphics[width=0.6\textwidth]{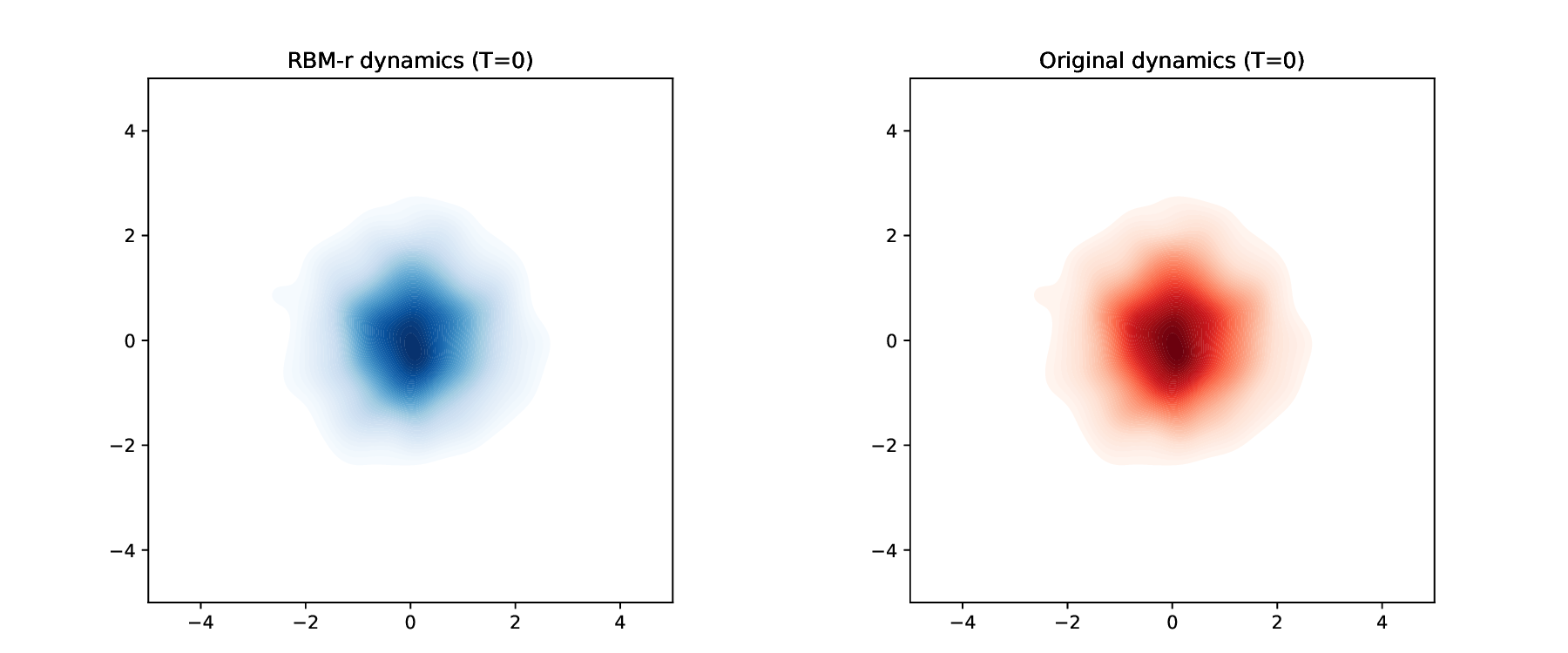}    

    \includegraphics[width=0.6\textwidth]{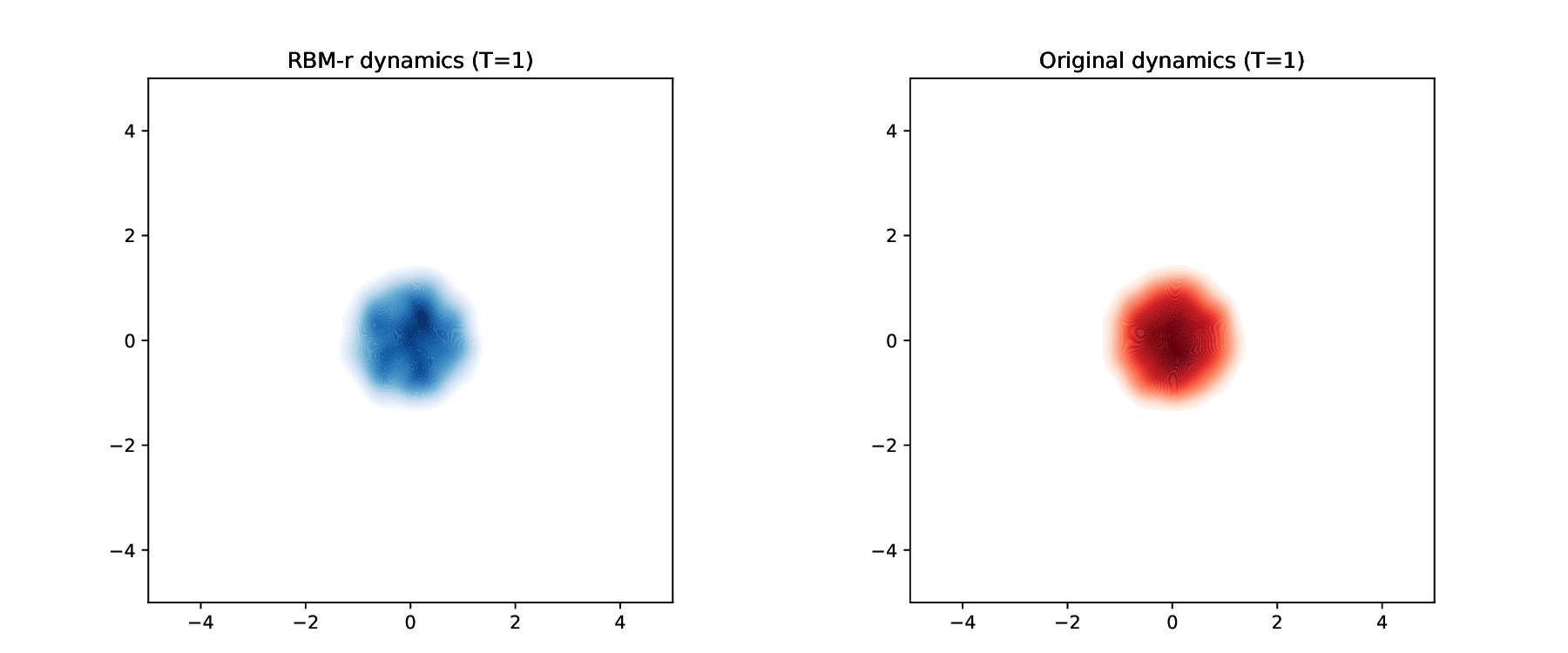}

    \includegraphics[width=0.6\textwidth]{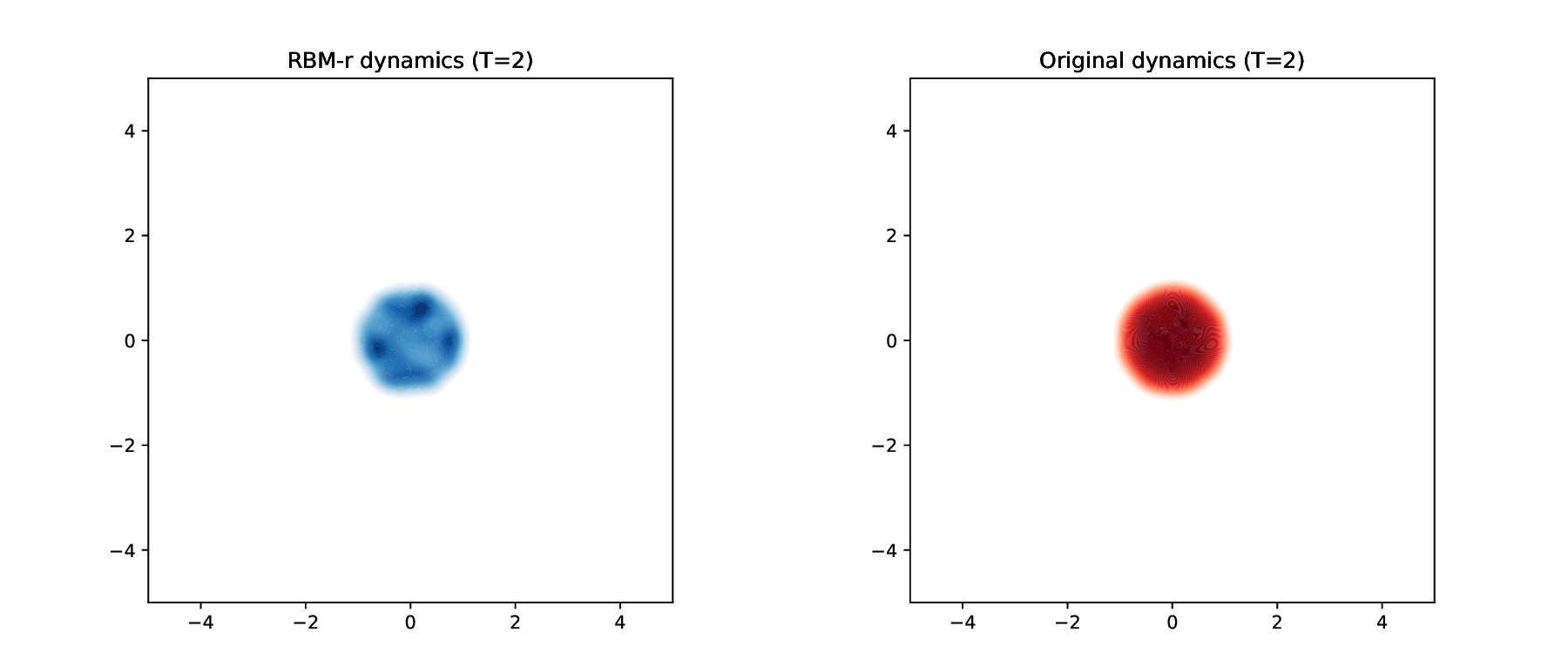}

\caption{Evolution of particle distributions of RBM-r dynamics and original dynamics}
\label{fig:distribution2}
\end{figure}

% \begin{figure}[!h]
% 	\centering
% 	\includegraphics[
% 	width=1.0\linewidth]{dyson1.eps}
% 	\caption{}
% 	\label{fig:dyson1}
% \end{figure}

% \begin{figure}[!h]
% 	\centering
% 	\includegraphics[
% 	width=1.0\linewidth]{dyson2.eps}
% 	\caption{}
% 	\label{fig:dyson1}
% \end{figure}

% \begin{figure}[!h]
% 	\centering
% 	\includegraphics[
% 	width=1.0\linewidth]{dyson3.eps}
% 	\caption{}
% 	\label{fig:dyson1}
% \end{figure}

As a final remark, there is still room to improve the theoretical Wasserstein convergence rate. For instance, for the time-discrete RBM dynamics, it is possible to obtain an improved convergence rate under the relative entropy. To our knowledge, existing experiments on verifying the convergence rate is limited to the strong error, which is only a sufficient condition for the Wasserstein error. So here in this section our main focus is on the empirical observation of the  efficiency and the convergence property of the RBM-r algorithm. We discuss more on possible methods to improve the theoretical convergence results in the next section.

\section{Conclusion and discussions}\label{sec:concludion}

In this paper, we prove the convergence of the Random Batch Method with replacement (RBM-r) for interacting particle systems proposed in \cite{jin2020random}, which is the same as the kinetic Monte Carlo (KMC) for the pairwise interacting particle system. In detail, via a sequence of probabilistic analysis, we prove in Theorem \ref{eq:thm:mainthm} that RBM-1 converges to the associated interacting particle system under Wasserstein-2 distance with an explicit rate. An improved rate is obtained when there is no diffusion in the system.

One natural extension of our method is to study the kinetic Monte Carlo applied to other dynamics like the stochastic Ising spin system \cite{bortz1975new} as discussed in Section \ref{sec:intro}. Indeed, as mentioned in \cite{voter2007introduction}, although people focus more on the equilibrium state in most research on KMC, the KMC algorithm is actually simulating some associated dynamical system. To our knowledge, the convergence of the KMC method to this underlying dynamics is not well-studied. Using our method based on random time change and $L^2$ wea law of large number, it is possible to provide a rigorous proof of KMC for suitable given models other than the first-order pairwise interacting particle system defined in \eqref{eq:IPSintro0}.

Another possible issue is whether our analysis is valid for more general form of volatility and noise, such as the multiplicative noise and noises with L\'evy jumps. To our knowledge, existing experiments show that RBM can still effectively simulate the interacting particle systems when the noise is multiplicative. However, currently it is difficult to extend our analysis to the multiplicative noise case. This is because the current proof relies on the constructed synchronous coupling, where the Brownian motions are the same in the coupled dynamics. When $\sigma$ is constant, the noise formally cancels out (i.e. there is no diffusion terms in $Z^i$ defined in equation \eqref{eq:Zidef}), which leads to the desired convergence. However, we would lose this fact under the multiplicative noise setting, so the proof would be rather difficult using the current coupling method. However, we do believe that the rigorous convergence analysis of RBM-r applied to systems with L\'evy noises can be established. Recently, the rigorous convergence of RBM-1 to this system has been proved \cite{liu2024random}, and we leave the extension to RBM-r as possible future work.

Other possible future work includes looking for advanced analysis tools to improve the current convergence rate from $O(\kappa^{\frac{1}{4}})$ to $O(\kappa^{\frac{1}{2}})$, especially when the volatility $\sigma$ is not zero. The current proof framework (based on coupling) brings essential challenges to improve the convergence rate, because when one looks at the trajectories of the particles, the random time change is naturally induced by the pseudo-time concept. The pseudo-time comes from the allowance of replacement when picking random batches, which is directly due to the construction of the RBM-r algorithm. Now since the random time change is unavoidable under the current coupling-based analysis framework, one always need to estimate the difference of the solution of a differential equation at (random) different time points. This then leads to the current convergence rate. However, we believe that it is possible to improve the convergence rate by considering the distributions of the particle system instead of their trajectories, and making use of the relative entropy (or KL-divergence). Recent literature has shown that this alternative framework is able to bring additional benefit. For instance, improved convergence rates in terms of the relative entropy for time discretization of various SDE systems with random drifts have been established using this framework \cite{li2022sharp, huang2024mean}. Remarkably, one key initial step of this alternative framework is to establish the Fokker-Planck-type equation that describes the time evolution of the distribution of the system we study. Then, similar (nontrivial) techniques from \cite{li2022sharp, huang2024mean} may be applied to compute and estimate the time derivative of the relative entropy. We leave the detailed analysis as a nontrivial future work.

\section*{Acknowledgements}
The research of J.-G. L. is partially supported by the National Science Foundation (NSF) under Grant No. DMS-2106988. 
The authors would like to acknowledge the financial support from the National Key R\&D Program of China, Project Number 2020YFA0712902, which enabled Zhenhao's visit to Duke University.
The authors would like to thank the editors and the anonymous reviewers for helpful comments and suggestions.

\appendix

\section{Proofs of Lemmas in Section \ref{sec:lemmas}}\label{sec:append}
In this section, we provide proofs of Lemmas \ref{lmm:moment}, \ref{lmm:hatX_continuity}, and \ref{eq:L2LLN}.

\subsection{Proof of Lemma \ref{lmm:moment}}\label{sec:lmmmoment}
\begin{proof}
We first estimate the process $X^{i}(t)$. By It\^o's formula,
\begin{equation*}
    \begin{aligned}
\frac{d}{d t} \mathbb{E}\left|X^i\right|^q= & q \mathbb{E}\left|X^i\right|^{q-2}\Big(-X^i \cdot \nabla V\left(X^i\right)+\frac{1}{N-1} \sum_{j: j \neq i} X^i \cdot K\left(X^i-X^j\right)\Big) \\
& +\frac{1}{2} q(q+d-2) \sigma^2 \mathbb{E}\left|X^i\right|^{q-2} .
\end{aligned}
\end{equation*}
By $\lambda$-convexity of $V(\cdot)$ in Assumption \ref{ass}, we have
$$
X^i \cdot \nabla V\left(X^i\right)=\left(X^i-0\right) \cdot\left(\nabla V\left(X^i\right)-\nabla V(0)\right)+X^i \cdot \nabla V(0) \geq \lambda \left|X^i\right|^2+X^i \cdot \nabla V(0) .
$$
Therefore, since $K$ is bounded by Assumption \ref{ass},
\begin{equation}\label{eq:momentforX}
\frac{d}{d t} \mathbb{E}\left|X^i\right|^q \leq-q \lambda \mathbb{E}\left|X^i\right|^q+q\left(\|K\|_{\infty}+|\nabla V(0)|\right) \mathbb{E}\left|X^i\right|^{q-1}+\frac{1}{2} q(q+d-2) \sigma^2 \mathbb{E}\left|X^i\right|^{q-2} .
\end{equation}
For the second term on the right-hand side of \eqref{eq:momentforX}, by Young's inequality, for any $\nu>0$,
$$
\mathbb{E}\left|X^i\right|^{q-1} \leq \frac{(q-1) \nu}{q} \mathbb{E}\left|X^i\right|^q+\frac{1}{q \nu^{q-1}}.
$$
For the last term on the right hand side of \eqref{eq:momentforX}, if $q=2$, it is bounded by $\frac{1}{2}q(q+d-2)$; otherwise, we can use Young's inequality to obtain a similar control for $\mathbb{E}\left|X^i \right|^{q-2}$: for any $\tilde{\nu}>0$,
\begin{equation*}
    \mathbb{E}\left|X^i \right|^{q-2} \leq \tilde{\nu}\mathbb{E}\left|X^i \right|^q + \frac{2}{q-2}\left(\frac{q-2}{q} \right)^{\frac{q}{2}} \tilde{\nu}^{-\frac{q-2}{2}}.
\end{equation*}
Consequently, by choosing small $\nu$ and $\tilde{\nu}$, we have
\begin{equation*}
    \frac{d}{d t} \mathbb{E}\left|X^i\right|^q \leq -C_1 \mathbb{E}\left|X^i\right|^q + C_2
\end{equation*}
for some $C_1$, $C_2 > 0$ independent of $t$, $N$, $p$, $i$. By Gr\"onwall's inequality, we conclude that
\begin{equation}\label{eq:finalmomentX}
    \sup_{1\leq i \leq N}\sup _{t \geq 0}\mathbb{E}\left|X^{i}(t)\right|^q \leq C_q^{(1)},
\end{equation}
where $C_q^{(1)}$ is a positive constant independent of $i$, $N$, $p$. As a direct corollary of \eqref{eq:finalmomentX}, due to the random time change relation (recall the definition for $\tau_n^i$ in \eqref{eq:defmni})
\begin{equation*}
    \hat{X}^{\ell,i}(\tau_n^i \kappa + t) = X^\ell(n\kappa + t), \quad \forall t\in [0,\kappa), \quad n = 0,1,2,\dots,\quad 1\leq \ell \leq N,
\end{equation*}
and since the random batch is independent of the Brownian motion (so the stopping time $\tau_n^i$ is independent of $\hat{X}$), we obtain the moment control for $\hat{X}^{\ell,i}$:
\begin{equation}
    \sup_{1\leq l,i \leq N}\sup _{t \geq 0}\mathbb{E}\left|\hat{X}^{\ell,i}(t)\right|^q \leq C_q^{(1)},   
\end{equation}
due to the random time change relation (recall the definition for $\tau_n^i$ in \eqref{eq:defmni})
\begin{equation*}
    \hat{X}^{\ell,i}(\tau_n^i \kappa + t) = X^\ell(n\kappa + t), \quad \forall t\in [0,\kappa), \quad n = 0,1,2,\dots,\quad 1\leq \ell \leq N.
\end{equation*}

Next, we estimate the process $\tilde{X}^i(t)$. Fix $i$. For $k$ such that $i \notin \mathcal{C}_k$ and $t \in [t_k,t_{k+1})$,
\begin{equation*}
    \frac{d}{dt}\mathbb{E}\left|\tilde{X}^i \right|^q = 0.
\end{equation*}
For $k$ such that $i \in \mathcal{C}_k$ and $t \in [t_k,t_{k+1})$,
\begin{equation*}
    \begin{aligned}
\frac{d}{d t} \mathbb{E}\left[\left|\tilde{X}^i\right|^q\mid \mathcal{F}_k\right]= & q \mathbb{E}\Big[\left|\tilde{X}^i\right|^{q-2}\Big(-\tilde{X}^i \cdot \nabla V\left(\tilde{X}^i\right)+\frac{1}{p-1} \sum_{j: j \neq i, j\in \mathcal{C}_k} \tilde{X}^i \cdot K\left(\tilde{X}^j-\tilde{X}^i\right)\Big)\mid \mathcal{F}_k\Big] \\
& +\frac{1}{2} q(q+d-2) \sigma^2 \mathbb{E}\left[\left|\tilde{X}^i\right|^{q-2}\Big| \mathcal{F}_k\right],
\end{aligned}
\end{equation*}
where $\mathcal{F}_k := \sigma\left(\tilde{X}^i(0), \tilde{B}^i(s), C_{k'}: s\leq t_k, k' \leq k, 1\leq i\leq N  \right)$. Using the similar estimates as for $X^i(t)$, we have
\begin{equation}\label{eq:importantinmoment}
    \frac{d}{dt}\mathbb{E}\left[\left|\tilde{X}^{i}\right|^q \mid \mathcal{F}_k\right]\leq -C_1' \mathbb{E}\left[\left|\tilde{X}^{i}\right|^q \mid \mathcal{F}_k\right] + C_2'
\end{equation}
for some $C_1'$, $C_2' > 0$ independent of $t$, $N$, $p$, $i$, $k$. Note that compared with the estimate for the process $X^i$, the only difference in deriving \eqref{eq:importantinmoment} is that, the summation $\frac{1}{N-1}\sum_{j\neq i}K(X^j - X^i)$ is replaced by $\frac{1}{p-1}\sum_{j\neq i,j\in \mathcal{C}_k}K(\tilde{X}^j - \tilde{X}^i)$. However, since we only use the boundedness of $K(\cdot)$ from Assumption \ref{ass}, their bounds are exactly the same. Now taking expectation about the randomness of $\mathcal{F}_k$ on both sides and using Gr\"onwall's inequality, we obtain
\begin{equation}
    \sup_{1\leq i \leq N}\sup _{t \geq 0}\mathbb{E}\left|\tilde{X}^i(t)\right|^q \leq C_q^{(2)},
\end{equation}
where $C_q^{(2)}$ is a positive constant independent of $i$, $N$, $p$. 

\end{proof}

\subsection{Proof of Lemma \ref{lmm:hatX_continuity}}\label{sec:lmmconti}

\begin{proof}

For the process $X$, using Assumption \ref{ass} and moment control in Lemma \ref{lmm:moment}, we have
\begin{equation*}
\begin{aligned}
    &\quad\mathbb{E}\left|X^\ell(t_2) - X^\ell(t_1) \right|^2\\
    &=\mathbb{E}\Big|-\int_{t_1}^{t_2}  \nabla V(X^l(s)) ds + \int_{t_1}^{t_2} \frac{1}{N-1}\sum_{j\neq \ell} K(X^j(s) - X^\ell(s)) ds + \sigma \int_{t_1}^{t_2}dB^\ell\Big|^2\\
    &\leq (t_2-t_1)\mathbb{E}\int_{t_1}^{t_2}\left(1 + |X^\ell(s)|^{2q} \right)ds + C' (t_2 - t_1)^2 + 3\sigma^2(t_2-t_1)\\
    &\leq C(t_2-t_1)^2 + 3\sigma^2(t_2-t_1)
\end{aligned}
\end{equation*}
for some positive $C'$, $C$ independent of $\ell$, $k$, $t_1$, $t_2$, $N$, $p$, $\sigma$. Note that we have used the polynomial growth condition for $\nabla V(\cdot)$ and boundedness for $K(\cdot)$ in assumption \ref{ass} in the first inequality, and the uniform-in-time moment bound in Lemma \ref{lmm:moment} has been applied for the second inequality above.

For the process $Z$, denote
\begin{equation}\label{eq:mathcalA}
    \mathcal{A} := \left(\cup_{k:i \in \mathcal{C}_k}[t_k,t_{k+1})\right)\cap[t_1,t_2],
\end{equation}
and recall that when $i \in \mathcal{C}_k$,
\begin{multline*}
    dZ^i(s) = - \left(\nabla V(\tilde{X}^i(s)) - \nabla V(\hat{X}^{i,i}(s)) \right) ds\\
    + \frac{1}{p-1}\sum_{j\in \mathcal{C}_k,j\neq i}K\left(\tilde{X}^j(s) - \tilde{X}^i(s) \right) ds - \frac{1}{N-1}\sum_{j\neq i}K\left(\hat{X}^{j,i}(s) - \hat{X}^{i,i}(s) \right) ds.
\end{multline*}
Similarly, since $K(\cdot)$ is bounded and $\nabla V(\cdot)$ grows at most polynomially by Assumption \ref{ass}, and using the uniform moment bound obtained in Lemma \ref{lmm:moment}, we have
\begin{equation*}
\begin{aligned}
    &\quad\mathbb{E}\left|Z^i(t_2) - Z^i(t_1) \right|^2\\
    &=\mathbb{E}\Big|-\int_{t_1}^{t_2} \textbf{1}_{\{s\in\mathcal{A}\}} \left(\nabla V(\tilde{X}^i(s))-\nabla V(\hat{X}^{i,i}(s))\right) ds\\
    &\quad+ \int_{t_1}^{t_2} \textbf{1}_{\{s\in\mathcal{A}\}} \left(\frac{1}{p-1}\sum_{j\in \mathcal{C}_k, j\neq i} K(\tilde{X}^j(s) - \tilde{X}^i(s)) - \frac{1}{N-1}\sum_{j\neq i} K(\hat{X}^{j,i}(s) - \hat{X}^{i,i}(s)) \right) ds \Big|^2\\
    &\leq (t_2-t_1)\mathbb{E}\int_{t_1}^{t_2}\left(1 + |\tilde{X}^i(s)|^{2q} + |\hat{X}^{i,i}(s)|^{2q}\right)ds + C' (t_2 - t_1)^2\\
    &\leq C(t_2-t_1)^2 
\end{aligned}
\end{equation*}
for some positive $C'$, $C$ independent of $i$, $\ell$, $k$, $t_1$, $t_2$, $N$, $p$, $\sigma$.

\end{proof}

\subsection{Proof of Lemma \ref{eq:L2LLN}}\label{sec:lmmlln}
\begin{proof}
Let $\tilde{\tau}_k^i := \tau_k^i - \tau_{k-1}^i$ for $k = 0,1,2,\dots$ (define $\tau_{-1}^i = 0$). Clearly, $\tilde{\tau}_k-1$ is the number of intervals that the particle with index $i$ remains outside the batch since the last chosen time. Consequently, the sequence $(\tilde{\tau}_k)_k$ are i.i.d. satisfying a geometric distribution
\begin{equation*}
    P(\tilde{\tau}_k = j) = \left(1 - \tfrac{p}{N} \right)^{j-1}\tfrac{p}{N}, \quad j = 1,2,\dots
\end{equation*}
with expectation $\mathbb{E}\tilde{\tau}_k^i = \tfrac{N}{p}$ and variance $Var(\tilde{\tau}_k^i) = \tfrac{N^2}{p^2} - \tfrac{N}{p}$.
Therefore, 
\begin{equation*}
\begin{aligned}
    \mathbb{E}\left|\kappa \tau_{n_t}^i - \tfrac{N}{p}t \right|^2 &= t^2\mathbb{E}\left|\frac{1}{n_t} \sum_{n=0}^{n_t} \left(\tilde{\tau}_n^i - \tfrac{N}{p}\right) \right|^2\\
    &= \kappa^2 \sum_{n=0}^{n_t} \mathbb{E}\left|\tilde{\tau}_n^i - \tfrac{N}{p} \right|^2 = Var(\tilde{\tau}_0^i)\kappa^2\left(t/\kappa + 1 \right).
\end{aligned}
\end{equation*}
\end{proof}

\bibliographystyle{plain}
\bibliography{main}

\begin{thebibliography}{10}

\bibitem{binder1973scaling}
K~Binder and E~Stoll.
\newblock Scaling theory for metastable states and their lifetimes.
\newblock {\em Physical Review Letters}, 31(1):47, 1973.

\bibitem{bortz1975new}
Alfred~B Bortz, Malvin~H Kalos, and Joel~L Lebowitz.
\newblock A new algorithm for {Monte} {Carlo} simulation of {Ising} spin
  systems.
\newblock {\em Journal of Computational Physics}, 17(1):10--18, 1975.

\bibitem{carrillo2021random}
Jos{\'e}~Antonio Carrillo, Shi Jin, and Yijia Tang.
\newblock Random batch particle methods for the homogeneous {Landau} equation.
\newblock {\em arXiv preprint arXiv:2110.06430}, 2021.

\bibitem{daus2022random}
Esther~S Daus, Markus Fellner, and Ansgar J{\"u}ngel.
\newblock Random-batch method for multi-species stochastic interacting particle
  systems.
\newblock {\em Journal of Computational Physics}, 463:111220, 2022.

\bibitem{franz2004finite}
Silvio Franz and Fabio~Lucio Toninelli.
\newblock Finite-range spin glasses in the {Kac} limit: free energy and local
  observables.
\newblock {\em Journal of Physics A: Mathematical and General}, 37(30):7433,
  2004.

\bibitem{frohlich1987some}
J{\"u}rg Fr{\"o}hlich and Bogus{\l}aw Zegarlinski.
\newblock Some comments on the {Sherrington-Kirkpatrick} model of spin glasses.
\newblock {\em Communications in mathematical physics}, 112:553--566, 1987.

\bibitem{golse2019random}
Fran{\c{c}}ois Golse, Shi Jin, and Thierry Paul.
\newblock The random batch method for $ n $-body quantum dynamics.
\newblock {\em arXiv preprint arXiv:1912.07424}, 2019.

\bibitem{greengard1987fast}
Leslie Greengard and Vladimir Rokhlin.
\newblock A fast algorithm for particle simulations.
\newblock {\em Journal of computational physics}, 73(2):325--348, 1987.

\bibitem{huang2024mean}
Zhenyu Huang, Shi Jin, and Lei Li.
\newblock Mean field error estimate of the random batch method for large
  interacting particle system.
\newblock {\em arXiv preprint arXiv:2403.08336}, 2024.

\bibitem{jabin2018quantitative}
Pierre-Emmanuel Jabin and Zhenfu Wang.
\newblock Quantitative estimates of propagation of chaos for stochastic systems
  with ${W}^{-1,\infty}$ kernels.
\newblock {\em Inventiones mathematicae}, 214:523--591, 2018.

\bibitem{jin2022mean}
Shi Jin and Lei Li.
\newblock On the mean field limit of the random batch method for interacting
  particle systems.
\newblock {\em Science China Mathematics}, pages 1--34, 2022.

\bibitem{jin2020random}
Shi Jin, Lei Li, and Jian-Guo Liu.
\newblock Random batch methods ({RBM}) for interacting particle systems.
\newblock {\em Journal of Computational Physics}, 400:108877, 2020.

\bibitem{jin2021convergence}
Shi Jin, Lei {Li}, and Jian-Guo Liu.
\newblock Convergence of the random batch method for interacting particles with
  disparate species and weights.
\newblock {\em SIAM Journal on Numerical Analysis}, 59(2):746--768, 2021.

\bibitem{jin2022random}
Shi Jin, Lei Li, and Yiqun Sun.
\newblock On the {Random} {Batch} {Method} for second order interacting
  particle systems.
\newblock {\em Multiscale Modeling \& Simulation}, 20(2):741--768, 2022.

\bibitem{jin2021random}
Shi Jin, Lei Li, Zhenli Xu, and Yue Zhao.
\newblock A random batch {Ewald} method for particle systems with {Coulomb}
  interactions.
\newblock {\em SIAM Journal on Scientific Computing}, 43(4):B937--B960, 2021.

\bibitem{jin2023ergodicity}
Shi Jin, Lei Li, Xuda Ye, and Zhennan Zhou.
\newblock Ergodicity and long-time behavior of the {Random} {Batch} {Method}
  for interacting particle systems.
\newblock {\em Mathematical Models and Methods in Applied Sciences},
  33(01):67--102, 2023.

\bibitem{jin2020randomquantum}
Shi Jin and Xiantao Li.
\newblock Random batch algorithms for quantum {Monte} {Carlo} simulations.
\newblock {\em arXiv preprint arXiv:2008.12990}, 2020.

\bibitem{li2022some}
Lei Li, Jian-Guo Liu, and Yijia Tang.
\newblock Some {Random} {Batch} {Particle} {Methods} for the
  {Poisson-Nernst-Planck} and {Poisson-Boltzmann} equations.
\newblock {\em Communications in Computational Physics}, 32(1), 2022.

\bibitem{li2022sharp}
Lei Li and Yuliang Wang.
\newblock A sharp uniform-in-time error estimate for {Stochastic} {Gradient}
  {Langevin} {Dynamics}.
\newblock {\em arXiv preprint arXiv:2207.09304}, 2022.

\bibitem{li2020random}
Lei Li, Zhenli Xu, and Yue Zhao.
\newblock A random-batch {Monte} {Carlo} method for many-body systems with
  singular kernels.
\newblock {\em SIAM Journal on Scientific Computing}, 42(3):A1486--A1509, 2020.

\bibitem{liang2022superscalability}
Jiuyang Liang, Pan Tan, Yue Zhao, Lei Li, Shi Jin, Liang Hong, and Zhenli Xu.
\newblock Superscalability of the random batch {Ewald} method.
\newblock {\em The Journal of Chemical Physics}, 156(1), 2022.

\bibitem{liu2024random}
Jian-Guo Liu and Yuliang Wang.
\newblock On {Random Batch Methods (RBM)} for interacting particle systems
  driven by {L\'evy} processes.
\newblock {\em arXiv preprint arXiv:2412.06291}, 2024.

\bibitem{metropolis1949monte}
Nicholas Metropolis and Stanislaw Ulam.
\newblock The monte carlo method.
\newblock {\em Journal of the American statistical association},
  44(247):335--341, 1949.

\bibitem{motsch2014heterophilious}
Sebastien Motsch and Eitan Tadmor.
\newblock Heterophilious dynamics enhances consensus.
\newblock {\em SIAM review}, 56(4):577--621, 2014.

\bibitem{mou2022improved}
Wenlong Mou, Nicolas Flammarion, Martin~J Wainwright, and Peter~L Bartlett.
\newblock Improved bounds for discretization of {Langevin} diffusions:
  Near-optimal rates without convexity.
\newblock {\em Bernoulli}, 28(3):1577--1601, 2022.

\bibitem{robbins1951stochastic}
Herbert Robbins and Sutton Monro.
\newblock A stochastic approximation method.
\newblock {\em The annals of mathematical statistics}, pages 400--407, 1951.

\bibitem{voter2007introduction}
Arthur~F Voter.
\newblock Introduction to the kinetic {Monte} {Carlo} method.
\newblock In {\em Radiation effects in solids}, pages 1--23. Springer, 2007.

\bibitem{wang2024deepparticle}
Zhongjian Wang, Jack Xin, and Zhiwen Zhang.
\newblock A deepparticle method for learning and generating aggregation
  patterns in multi-dimensional {Keller--Segel} chemotaxis systems.
\newblock {\em Physica D: Nonlinear Phenomena}, 460:134082, 2024.

\bibitem{ye2024error}
Xuda Ye and Zhennan Zhou.
\newblock Error analysis of time-discrete random batch method for interacting
  particle systems and associated mean-field limits.
\newblock {\em IMA Journal of Numerical Analysis}, 44(3):1660--1698, 2024.

\end{thebibliography}

%    Text of article.

%    Bibliographies can be prepared with BibTeX using amsplain,
%    amsalpha, or (for "historical" overviews) natbib style.
%\bibliographystyle{amsplain}
%    Insert the bibliography data here.

\end{document}